\documentclass{amsart}

\usepackage{amsthm,amssymb,verbatim,color}
\usepackage{graphicx}
\usepackage{enumerate}
\usepackage{enumerate}
\usepackage{hyperref}

\usepackage[alphabetic, msc-links]{amsrefs}

\newcommand{\sd}[1]{\frac{d}{d #1}}
\newcommand{\pd}[2]{\frac{\partial#1}{\partial #2}}

\newcommand{\ip}[2]{\left\langle{#1},{#2} \right\rangle}

 \newcommand{\norm}[1]{\left\lvert{#1}\right\rvert}

    \newtheorem{theorem}    {Theorem}     
    \newtheorem{conjecture} {Conjecture}
     
    \newtheorem{lemma}   [theorem]       {Lemma}
    \newtheorem{corollary} [theorem] {Corollary}

    \newtheorem*{claim*}{Claim}

\newcommand{\starref}[1]{\textasteriskcentered}

    \theoremstyle{definition}
    \newtheorem{definition}  [theorem] {Definition}
    \newtheorem{assumption}  [theorem] {Assumption}

    \theoremstyle{definition}
    \newtheorem{remark}   [theorem]       {Remark}
    \newtheorem*{starenv*}{\textasteriskcentered}

\newcounter{starlabel}
\newcommand{\starlabel}[1]{%
  \refstepcounter{starlabel}%
  \label{#1}%
}

\begin{document}

  \newcommand{\marginpor}[1]{}



\title[Backwards uniqueness with AC singularities]{Backwards uniqueness for Mean curvature flow with asymptotically conical singularities}
\author{J.\,M.\, Daniels Holgate}
\author{Or Hershkovits}


\dedicatory{Dedicated to the memory of Professor Shmuel Agmon}

\maketitle

\begin{abstract}
In this paper we demonstrate  that if two mean curvature flows of compact hypersurfaces $M^1_t$ and $M^2_t$ encounter only isolated, multiplicity one, asymptotically conical singularities at the first singular time $T$, and if $M^1_T=M^2_T$ then $M^1_t=M^2_t$ for every $t\in [0,T]$.  This is seemingly the first backwards uniqueness result for any  geometric flow with singularities, that  assumes neither {self-shrinking} nor global asymptotically conical {behaviour}. This necessitates the development of new global tools to deal with  both the core of the singularity, its asymptotic structure, and the smooth part of the flows simultaneously.  { As an immediate application, we show that low entropy flows in $\mathbb{R}^4$ are backwards unique}.
\end{abstract}

\maketitle

\section{introduction}

Given a smooth compact hypersurface in $M\subseteq \mathbb{R}^{n+1}$, standard parabolic PDE theory implies that there exists some $t_0>0$ and a smooth family of hypersurfaces $\{M_t\}_{t\in [0,t_0)}$  such that $M_0=M$ and for every time $t\in [0,t_0)$, the normal motion of the position vector is given by the mean curvature vector of $M_t$:
\begin{equation}
\left(\frac{dx}{dt}\right)^{\perp}=\vec{H}.
\end{equation}
Such a family is called a mean curvature flow (MCF) from $M$. {Standard theory further} implies that the smooth evolution $\{M_t\}_{t\in [0,T_0)}$ is unique, and one can refer to \textit{the} MCF.  The standard maximal time of existence criteria manifests in this context as the existence of a \textit{finite} time $T_{0}<\infty$ at which the second fundamental form of $M_t$ explodes. 

\bigskip

Whilst classical uniqueness of solutions to non-linear parabolic PDEs is a central pillar of the theory, there is another, substantially more subtle (although admittedly less central) notion of uniqueness - the notion of backwards uniqueness.   In the context of MCF, it was established in \cite{BU} that if $M^1_t$ and $M^2_t$ are two \textit{smooth} mean curvature flows on the interval $[0,T]$ satisfying $M^1_{T}=M^2_{T}$ then $M^1_0=M^2_0$. This proof relied on ideas found in \cite{LP}, reverting to an implementation for vector bundles appearing in \cite{KS_unique1}.

\bigskip

Our discussion so far has only considered \textit{smooth} mean curvature flow. However, it has long been established that mean curvature flow does not necessarily vanish completely at its maximal time of smooth existence $T_0$, as demonstrated by the dumbbell surface. The study of the formation of singularities in MCF, as well as the regularity and structure theory of weak solution extending the flow past time $T_0$  has been the central theme of research in MCF for the last three decades.

\bigskip

The existence question {for} various notions of flow past singularities was addressed  in \cite{ES,CGG, Ilm_elip, Brakke}. The weak mean curvature flow is {non-}vanishing for a maximal time $T_{\textrm{ext}}<\infty$ called the \textbf{extinction time}. Typically{,} $T_{\mathrm{ext}}$ is strictly larger than the maximal time of smooth evolution $T_0$, which is called the \textbf{first singular time}. In the very same papers where existence was discussed, the appropriate notion of non-uniqueness was pin pointed. While it was clear from the beginning that weak mean curvature flow starting from a \textit{singular} initial condition might evolve {non-}uniquely (i.e fatten), it took another decade until an example of a {non-}unique evolution past singularities of an initially compact smooth hypersurface  was constructed \cite{White_ICM, Ilmanen_White} (a more recent construction of non-uniqueness follows from  \cite{Ketover,LeeAtAll,CHS}). Even when the flow is {non-}unique, there are still two distinguished weak mean curvature flows - the \textit{outer most flow} and the \textit{inner most flow}, and uniqueness is best characterized by both of them being equal to the maximal weak flow, which is called the level set flow \cite{HW}.   

\bigskip

The purpose of this article is to initiate the study of backwards uniqueness for mean curvature flow past singularities. We propose the following conjecture which motivated the current research:

\begin{conjecture}[Backwards uniqueness conjecture past singularities for MCF in $\mathbb{R}^3$]\label{backwards_uniqueness_conjecture}
Let $M^i \subset \mathbb{R}^3$, $i=1,2$ be two smooth, connected, compact  hypersurfaces in $\mathbb{R}^3$. Denote by $M^i_t$ the outermost mean curvature flows emanating from $M^i$. If there exists a time $t_0$ which is strictly prior to the extinction time of $M^1_t$, such that $M^1_{t_0}=M^2_{t_0}$ then 
\begin{equation}
M^1_t\equiv M^2_t,\qquad \forall t\in [0,t_0]. 
\end{equation} 
\end{conjecture}

{At present, it appears} the only existing results about backwards uniqueness of \textit{singular} MCF or other singular geometric flows, prior to the ones which appear in this  paper,  require both self-similarity \textit{and} an asymptotic assumption \cites{Wang_unique,Wang_unique_Cyl,KW,KW2, Kot_cone}.\footnote{Full self-similarity can sometime be deduced \textit{a posteriori}. See \cite{Kot_cone}.}  In particular, the existing results do not apply  to flows starting from compact surfaces. 

\bigskip

In order to motivate the above conjecture, we need to discuss MCF singularities. The study of singularity formation for MCF, which is {closely} connected with the regularity theory of the weak {MCF}, was initiated by Huisken \cite{Huisken_mono}, utilizing his  monotonicity formula. {Said formula demonstrates} that singularities of MCF are modelled on \textit{self-shrinking} flows: solutions of the mean curvature flow satisfying of the form $(\sqrt{-t}\Sigma)_{t\in (-\infty,0)}$, where $\Sigma$, {known as a} \textbf{self-shrinker}, satisfies the elliptic equation
\begin{equation}
\vec{H}+\frac{x^{\perp}}{2}=0.
\end{equation}   
Importantly, for mean curvature flow of surfaces in $\mathbb{R}^3$, it follows from \cite{Ilm_surfaces, Bamler_Kleiner,Wang_asym, CS, CM,Schulze_cpt} that if $(x_0,t_0)$ is a singular point at the first singular time, then either
\begin{enumerate}[I.]
\item ({C}ylindrical singularity)  There exists some orthogonal transformation $O$ such that 
\begin{equation*}
\lim_{t\rightarrow t_0}\frac{M_t-x_0}{\sqrt{t_0-t}}= O(S^{1}(\sqrt{2})\times \mathbb{R}),\;\textrm{or}
\end{equation*}
\item\label{AC_sing_alt} (Asymptotically conical singularity) there exists a {self-shrinker} $\Sigma\subseteq \mathbb{R}^{3}$ and a cone $C \subseteq \mathbb{R}^{3}$ such that $C-\{0\}$ is smooth and 
\begin{equation}\label{AC_eq}
\lim_{t\rightarrow t_0}\frac{M_t-x_0}{\sqrt{t_0-t}}=\Sigma, \qquad \lim_{\lambda\rightarrow 0}\Sigma =C \;\;\textrm{in}\; C^{\infty}_{\textrm{loc}}(\mathbb{R}^{n+1}-\{0\}), \; \textrm{{or}}
\end{equation} 
\item (Compact singularity)  There exists a compact self-shrinker $\Sigma\subseteq \mathbb{R}^3$ such that
\begin{equation*}
\lim_{t\rightarrow t_0} \frac{M_t-x_0}{\sqrt{t_0-t}}=\Sigma, \; \textrm{or}
\end{equation*}
\item\label{cyl_end} (Non-cylindrical singularities  with cylindrical ends) There exists a {self-shrinker} $\Sigma$, such that $\Sigma$ has some cylindrical end, but $\Sigma$ is not the cylinder $S^1(\sqrt{2})\times \mathbb{R}$, and a sequence of times $t_j\rightarrow t_0$ such that 
\begin{equation*}
\lim_{j\rightarrow \infty} \frac{M_{t_j}-x_0}{\sqrt{t_0-t_j}}=\Sigma.
\end{equation*}
\end{enumerate}
{The above limits} are to be interpreted in the $C^{\infty}_{\mathrm{loc}}$ sense, i.e. as $t\rightarrow T$ larger and larger portions of the surfaces on the left-hand side can be written as graphs of $C^{\infty}$ function, with smaller and smaller $C^k$ norm, over the corresponding balls in the surfaces in the right-hand side. 

\begin{remark}[On Backwards uniqueness and singularity analysis]
It is a famous conjecture of Ilmanen \cite{Ilmanen_problems} (the ``no cylinder conjecture'') that alternative IV above {cannot} happen. In fact, while the discussion above (and also \cites{White_strat, CM_structure, CM_loj, ADS,BC,CHH,CCMS, CM_generic,Holgate, CHS, BH,Brendle_genus, BH_surgery, HK_surgery, Wang_unique}) is a testament for the deep understanding we currently posses for MCF singularities in $\mathbb{R}^3$, the no {cylinder} conjecture is but one of several open problems regarding singularities of mean curvature flow of surfaces. Some of these problems, including the ``no cylinder'' conjecture, are {intimately} related to backwards uniqueness question, \cite{Wang_unique_Cyl}. Interestingly, in other parabolic settings, such as the three dimensional Navier--Stokes equation, backwards uniqueness forms a key tool in the study of singularity formation. Indeed, this was the main motivation for \cite{ESS}, on which \cite{Wang_unique} relies. This provides another motivation, in addition to the fundamental one, to the study of Conjecture \ref{backwards_uniqueness_conjecture}.            
\end{remark}
 
\bigskip 

As our main result holds in any dimension, let us make {the }definition of an asymptotically conical (AC) singularity concrete:

\begin{definition}[Asymptotically conical singularity]\label{AC_forwards}
We say that a MCF $(M_t)_{t\in (0,t_0)}$ in $\mathbb{R}^{n+1}$ has an asymptotically conical singularity at $(x_0,t_0)$ if there exists a smooth {self-shrinker} $\Sigma\subseteq \mathbb{R}^{n+1}$ and a cone $C \subseteq \mathbb{R}^{n+1}$ such that \eqref{AC_eq} holds.
\end{definition}

\begin{remark}\label{con_neigh}
If $(M_t)_{t\in [0,t_0)}$ has an asymptotically conical singularity at $(x_0,t_0)$ then by \cite{White_strat} there exists a ball $B(x_0,\varepsilon)$ and a smooth hypersurface $N\subseteq B(x_0,\varepsilon)-\{x_0\}$ such that $M_t\rightarrow N$ in $B(x_0,\varepsilon)-\{x_0\}$ as $t\rightarrow t_0$. In particular, if $(x_0,t_0)$ is a conical singularity, and it's the unique singular point at time $t_0$, then we can extend $M_t$ to a limit hypersurface $M_{t_0}$, which is smooth on $\mathbb{R}^{n+1}-\{x_0\}$. See \cite{CS} for a detailed explanation.
\end{remark}

We can now state the main theorem of this paper.

\begin{theorem}[Backward uniqueness at isolated asymptotically conical singularities] \label{main_thm}
Let $(M^i_t)_{t\in [0,t_0]}$ for $i=1,2$ be two mean curvature flows in $\mathbb{R}^{n+1}$, smooth in $\mathbb{R}^{n+1}\times [0,t_0] - \{(0,t_0)]\}$, having an AC singularity at $(0,t_0)$.  If $M^1_{t_0}=M^2_{t_0}$ then $M^1_t=M^2_t$ for every $t\in [0,t_0)$.
\end{theorem}   

\begin{remark}\label{several_sings}
For the sake of brevity, the theorem is stated with a single AC singularity at time $t_0$. The theorem and its proof hold for any number (necessarily  finite, by remark \ref{con_neigh}) of conical singularities.  
\end{remark}

\begin{remark}
When $n=2$, { or when and $n=3$ and the Gaussian entropies (see \cite{CM_generic}) of the $M^i_t$ are smaller than the entropy of $\mathbb{S}^2\times \mathbb{R}$}, if $M^1_t$ has an AC singularity at $(x_0,t_0)$ and $M^2_{t_0}=M^1_{t_0}$ then $M^2_t$ has an AC singularity at $(x_0,t_0)$ as well. See Lemma \ref{AC_from_other} in the preliminary section.
\end{remark}

Theorem \ref{main_thm} substantiates Conjecture \ref{backwards_uniqueness_conjecture}, in that it {attests} that one out of the  four (conjecturally three) types of singularities for MCF for surfaces is not the cause of backwards non-uniqueness. The case of cylindrical singularities will be addressed in a subsequent work of the authors. The appropriate notion for backwards uniqueness at compact singularities was addressed in \cite{DHH}, which will be discussed further below. { To further illustrate the relation of Theorem \ref{main_thm} to Conjecture \ref{backwards_uniqueness_conjecture}, we include the following easy analogue of the conjecture.

\begin{corollary}[Backwards uniqueness of low entropy flows in $\mathbb{R}^4$]
\label{backwards_uniqueness_cor}
Let $M^i \subset \mathbb{R}^4$, $i=1,2$ be two smooth, connected, compact  hypersurfaces in $\mathbb{R}^4$ with Gaussian entropy lower than the entropy of the cylinder $\mathbb{S}^2\times \mathbb{R}$. Denote by $M^i_t$ the outermost mean curvature flows emanating from $M^i$. If there exists a time $t_0$ which is strictly prior to the extinction time of $M^1_t$, such that $M^1_{t_0}=M^2_{t_0}$ then 
\begin{equation}
M^1_t\equiv M^2_t,\qquad \forall t\in [0,t_0]. 
\end{equation} 
\end{corollary} 
}

\begin{remark}[Contrasting uniqueness with backwards uniqueness]\label{contrast} Theorem \ref{main_thm} is somewhat  surprising when comparing uniqueness and backwards uniqueness: The backwards flow of any parabolic equation is less stable  than the forward flow, rendering backwards uniqueness a more subtle issue than forward uniqueness. Nevertheless AC singularities do cause non-uniqueness \cite{Ketover,LeeAtAll,CHS,Ilmanen_White}, while Theorem \ref{main_thm} postulates that they do not hinder backwards uniqueness. In fact, the picture emerging from \cites{HW, CHS, CHH,Bamler_Kleiner} and the aforementioned ``no cylinder'' conjecture is that AC singularities are the source, and the only source, of forward {non-}uniqueness. This further corroborates the validity of Conjecture \ref{backwards_uniqueness_conjecture}  
\end{remark}

\bigskip

\subsection{{Comparison} to previous works} Before sketching the proof of Theorem \ref{main_thm}, let us explain what new difficulties it poses by comparing it to existing backwards uniqueness results. This will also serve to direct us towards the proof.

\bigskip

The most elementary method for proving backwards uniqueness for {non-linear} heat type equation{s} is {via} ``frequency function'' argument{s}. This approach, {initiated} by Agmon and Nirenberg in \cites{AN1,AN2}, was successfully used for backwards uniqueness of smooth MCF (and other geometric evolution equations) in \cite{KS_simple}. The regularity of the flow was crucially used in showing that the terms in the analysis resulting {from} the change of the heat type operator {over time}, as well as the {non-linearities}, could be dominated by the main {``Rayleigh quotient''} plus a constant. This requires pointwise (in space) bounds for these terms.  More recently, using a frequency function analysis for the height of one flow over another (in Gaussian space), we were able to show in \cite{DHH} that two mean curvature flows $M^1_t,M^2_t$ which experience a multiplicity one compact singularity (as in alternative III above) at $(x_0,t_0)$ and satisfy 
\[
\lim_{t\rightarrow t_0}\frac{d_{H}(M^1_t,M^2_t)}{(t_0-t)^k}=0 \qquad \forall k=1,2,\ldots 
\]  
must satisfy $M^1_t=M^2_t$ for every $t\in [0,t_0)$ (see also \cites{MartinHagemayerSesum,Ks}). To control the errors and change in operator, we needed a {$L^1$} in time, but \textbf{global} $C^2$ in space, bound for the shrinker quantity (in rescaled space time). Such bounds were available since in the case of \cite{DHH}, when $t$ is close to $t_0$, the entirety of  $M^1_t$ is modelled on a self-shrinker.  Evidently, such bounds are not available in the context of Theorem \ref{main_thm}.

{Considering} whether such difficulties could be {overcome}, we remark that frequency function argument do not seem suitable for combining different regimes, and that even if such combination were possible, there is an intermediate conical regime between the shrinking regime and the smooth regime, where neither smooth analysis nor blown up shrinking analysis is relevant.

\begin{remark}[Best case scenario for a frequency analysis] Let us mention that even if such ``combined frequency analysis'' were possible in this setting, the discussion above indicated that it would only give half of the picture, in the form of a \textit{conditional uniqueness} result: that if $M^1_t$ and $M^2_t$ are sufficiently close as $t\rightarrow t_0$ then $M^1_t=M^2_t$. {Establishing that this closeness is achieved would require another} argument.  
\end{remark}

\bigskip

The second standard method for proving backwards uniqueness is via {Carleman estimates}, which are {a} form of weighted energy estimates with {convex} kernels. As was already mentioned, following \cite{LP}, Kotschwar \cite{KS_unique1} had used this approach to derive backwards uniqueness for smooth Ricci flows. This was later adapted to the MCF in \cite{BU}. While ideas from \cite{LP} do play a key role in our argument, there are several key difficulties, in increasing order of importance.
\begin{enumerate}[(A)]
\item The polynomial kernel of \cite{LP} is not a {good} fit for dealing with the singularity: The terms coming from the singularity are of higher order than the good terms from the kernel.
\item The singularity does not allow for the ``shifted polynomial weight'' trick in \cite{LP}. Thus, any argument based on \cite{LP} can only be expected to provide ``conditional uniqueness'', which should be augmented by an appropriate ``rate of decay'', {coming} from a {separate} argument.
\item While it is possible to obtain {a} conditional backwards uniqueness result by following the energy method of \cite{LP} {and} choosing an appropriate weight, the degree of closeness it requires is far too stringent for it to be useful: {adapting the} \cite{LP} argument {accordingly yields the existence of} some $\beta<\infty$, depending on the model $\Sigma$ of the AC singularity, such that if $M^1_t$ and $M^2_t$ are closer than $\exp\left(-(t_0-t)^{-\beta}\right)$, they must coincide. There is no reason for $\beta=\beta(\Sigma)$ to be small, and obtaining the required decay rate via some other argument,  with any specific $\beta<\infty$ (indeed, with any  $\beta>1$) seems unrealistic.
 \end{enumerate}  

\bigskip

Using much more involved Carleman estimates, following \cite{ESS}, a seminal work of Wang \cite{Wang_unique} shows that two {self-shrinker}s which have the same asymptotically conical end must themselves coincide. While ideas from Wang do play a key role here, we should again mention two important differences in increasing level of importance.
\begin{enumerate}[(a)]
\item {I}n \cite{Wang_unique} one assumes a global {self-shrinking} behaviour, as well as a conical behaviour of the end. We have no such global information, and there is no reason to assume even an approximate {self-shrinking} behaviour extends up to the conical scale.  
\item More importantly, the argument in \cite{Wang_unique} uses {Carleman} estimates along the conical end, without addressing what happens at the core of the shrinker. The Carleman methods of \cite{ESS} {used in} \cite{Wang_unique} are linear estimates, which eventually lead to the conclusion that {a} (non-rapidly growing) solution to the heat equation (with {controlled} errors) on $\mathbb{R}^n-B(0,1)$, {that} is zero at the final time, had been zero all along, without specifying {Dirichlet} data on $\partial B(0,1)$. This is simply not true on a compact space minus a ball, as the heat equation on a domain with smooth boundary is  {\textit{boundary null-controllable}}: given any initial data, one can prescribe a Dirichlet data such that solution becomes zero in finite time \cite{Coron_book}.

Thus, if we manage to adapt \cite{Wang_unique}, the best it could give is a rate of decay of $M^2_t$ to $M^1_t$, away from the balls $B(0,R\sqrt{t_0-t})$ for $R$ sufficiently large. This should be complemented by a rate on decay inside these balls, as well as compatible conditional uniqueness result, both requiring a different argument.
   
\end{enumerate}  

\subsection{Proof outline}

As is {indicated} by the discussion above, our proof of Theorem \ref{main_thm} consists of two {independent} steps.  { The first is our }``conditional uniqueness'' result, {that asserts a rate at which if two flows converge faster than, then they} must in fact coincide, {and a second that shows this rate is always attained}. As {our} paper deals with backwards uniqueness, it is both more {convenient} and customary {to the backwards-{uniqueness} literature} to consider backwards mean curvature flow (BMCF). {Henceforth,} we will consider hypersurfaces $(M_t)_{t\in [0,T]}$, smooth except when $t=0$ and $x=0$, evolving by
\begin{equation}
\left(\frac{dx}{dt}\right)^{\perp}=-\vec{H},
\end{equation}
such that the MCF $t\mapsto M_{T-t}$ has an isolated conical singularity at $(0,T)$. 

\bigskip 

Our first theorem states
\begin{theorem}[{Rate of convergence}]\label{thm_rates}
Let $(M^i_t)_{t\in [0,T]}$, $i=1,2$  be two BMCFs in $\mathbb{R}^{n+1}$, both smooth outside of $(0,0)$ and with AC singularity at $(0,0)$. We moreover assume that $M^2_0=M^1_0$. Then there exits some $\beta\in (0,1)$ such that
\begin{equation}\label{prox_conc}
d_H(M^1_t,M^2_t)\exp(t^{-\beta})\rightarrow 0, \qquad \textrm{as}\; t\rightarrow 0,
\end{equation}
\end{theorem}

This is complemented by the following theorem 

\begin{theorem}[{Conditional Uniqueness}]\label{thm_cond_back}
Let $(M^i_t)_{t\in [0,T]}$ $i=1,2$ be two BMCFs on $\mathbb{R}^{n+1}$, both smooth outside $(0,0)$ with AC singularity at $(0,0)$. If there exists some $\beta>0$  for which 
\begin{equation}\label{prox_ass}
d_H(M^1_t,M^2_t)\exp(t^{-\beta})\rightarrow 0, \qquad \textrm{as}\; t\rightarrow 0,
\end{equation}
then $M^1_t=M^2_t$ for every $t$.
\end{theorem} 
{
For the sake of completeness, we record that combining Theorem \ref{thm_cond_back} and Theorem \ref{thm_rates}  immediately yields a proof of our main theorem.
\begin{proof}[Proof of Theorem \ref{main_thm} (Backward uniqueness at isolated AC singularities)]
Abusing notations by letting $M^i_t$ be the BMCF corresponding to the MCF $M^i_t$ in the statement (i.e the new $M^i_t$ is given by $M^i_{T-t}$ for the original MCFs), Theorem \ref{thm_rates} (Rate of convergence) implies that the assumption \eqref{prox_ass} holds true. Thus, applying Theorem \ref{thm_cond_back} (Conditional Uniqueness) we get that $M^1_t=M^2_t$ for every $t\in [0,T]$.
\end{proof}
}
We now sketch the proof of {these} theorems, in opposite order. {We start by noting} the condition $M^2_0=M^1_0$, shared by both theorems, implies that $M^2_t$ {can be written} as a graph of a function $u$ over $M_t:=M^1_t$, perhaps after reducing $T$. This function $u$ is approximately a solution to the linearization of backwards MCF $\tilde{L}u\cong 0$, where $\tilde{L}u=L +|A|^2u$ and $L$ is the backwards heat operator
\begin{equation}
Lu= \partial_tu+ \Delta u.
\end{equation}
This follows from the standard theory (and in particular \cites{CS, Wang_unique}), and the details are discussed in Section \ref{prelim_sec}.

\bigskip

In order to prove Theorem \ref{thm_cond_back} (Conditional Uniqueness), we strive to obtain a global Carleman type inequality, bounding {the} weighted $W^{1,2}$ norm of $u$ via {the} weighted $L^2$ norm of $Lu$, {following the ideas} of \cite {LP}.\footnote{In this setting, the reaction terms is not the main source of difficulty, so there is no loss estimating via $L$ rather than $\tilde{L}$.} Importantly, in order to absorb the singularity, polynomial weights do not {suffice}, which {is reflected} in the exponential type decay imposed in \eqref{prox_ass}. 

The key estimate leading to Theorem \ref{thm_cond_back}  {states, for a function $w$ on $M_t$, that satisfies $w(\cdot,T)=0$, $w(\cdot,0)=0$, and decays to zero in $C^1$ faster than $\exp(-t^{-\beta_1})$, that there exists a $\varepsilon>0$ such that} for $\beta\in (0,\beta_1)$ and any $\alpha$ sufficiently large     

\begin{equation}\label{key_est}
\|Lw\|^2_{\mathcal{G}_{\alpha,\beta}} \geq  \|t^{-1-\beta/2}w\|^2_{\mathcal{G}_{\alpha,\beta}}+\varepsilon \|t^{-1/2}w\|^2_{\mathcal{G}_{\alpha,\beta},1},
\end{equation}

where 
\begin{equation}
\mathcal{G}=\mathcal{G}_{\alpha,\beta}=\exp\left(\frac{\alpha}{t^\beta}\right),\;
\textrm{and}\qquad\|w\|^2_{\mathcal{G}}=\int_{0}^T\int_{M_t}w^2\mathcal{G},\qquad \|w\|^2_{\mathcal{G},1}=\int_{0}^T\int_{M_t}|\nabla w|^2\mathcal{G}.
\end{equation}
See Theorem \ref{smooth_global_Carleman}. While estimate \eqref{key_est} is very similar to the one obtained in \cite{LP}, the methods therein {can} only produce it if $\beta$, and hence $\beta_1$, is sufficiently large (depending on the AC singularity model). Such  an estimate with {a} \textit{specific} $\beta<\infty$ is completely useless, since obtaining such rates a priori is not feasible. Indeed, any local, quasi-linear based argument using that $M^2_0=M^1_0$ could not produce $\beta_1>1$, and our rate of decay theorem, Theorem \ref{thm_rates} only produces some, possibly very small $\beta_1>0$.

Working with arbitrary $\beta$ (and $\beta_1$), the strategy of \cite{LP} does lead to an inequality similar to \eqref{key_est}, but with additional, possibly big, negative weighted $W^{1,2}$ terms on the right-hand side, which are supported on a (scale invariant) shrinking ball $I^r_t:=\{|x|\leq r\sqrt{t}\}$ with $r<\infty$. Loosely speaking, the obtained estimate (See Theorem \ref{smooth_global_Carleman_neg_term}) {states}   

\begin{align}\label{key_est_minus}
\|Lw\|^2_{\mathcal{G}_{\alpha,\beta}} \geq  &\|t^{-1-\beta/2}w\|^2_{\mathcal{G}_{\alpha,\beta}}+\|t^{-1/2}w\|^2_{\mathcal{G}_{\alpha,\beta},1}\\
-C&\|t^{-1-\beta/2}w\|^2_{I^r_t,\mathcal{G}_{\alpha,\beta}}-C\|t^{-1/2}w\|^2_{I^r_t,\mathcal{G}_{\alpha,\beta},1},\nonumber
\end{align}
where 
\begin{equation}
\|w\|^2_{I^r_t,\mathcal{G}}=\int_{0}^T\int_{M_t\cap I^r_t}w^2\mathcal{G},\qquad \|w\|^2_{I^r_t,\mathcal{G},1}=\int_{0}^T\int_{M_t\cap I^r_t}|\nabla w|^2\mathcal{G}.
\end{equation}
The negative terms on the {right-hand} side are due to the evolution of the metric and the volume form. The conicality and smoothness allow {us} to absorb these bad terms outside of $I^r_t$ when $r$ is large, but considering the (standard) volume  inside the region $I^r_t$ does not allow one to utilize that $M_t$ {is close to being} self-shrinking, leading to these negative terms.

\bigskip

To complement this estimate, we take $R\gg r$ and use the shrinker-like behaviour of $M_t$ inside $I^R_t$. Taking weights which combine the Gaussian, polynomials and the $\mathcal{G}$, and arguing as in \cite{LP} with the Gaussian volume and these weights,  we get  that  if $\bar{w}$ is obtained by cutting off $w$ between $I^R_t$ and $I^{R+1}_t$ then,  roughly speaking

\begin{equation}\label{Car_res_cut}
\left\|L\bar{w} \right\|^2_{I^{R+1}_t,\Phi\mathcal{G}} \geq \left\|t^{-1-\beta/2}w \right\|^2_{I^R_t,\Phi \mathcal{G}} +\left\|t^{-1/2}w\right\|^2_{I^{R}_t,\Phi\mathcal{G}},
\end{equation}
where $\Phi(x,t)=\exp(-|x|^2/4)$. This is done{ at} the level of the rescaled flow in Theorem \ref{Core_Car}, which implies \eqref{Car_res_cut} {after} changing variables. It is important that the Gaussian estimate \eqref{Car_res_cut} is applied to  $\tilde{w}$ and not $w$:  the contributions from the region were $M_t$ is not self-shrinking could not have been absorbed in such an argument.

\bigskip

To obtain \eqref{key_est} (which is done in Theorem \ref{smooth_global_Carleman}), we multiply \eqref{Car_res_cut} by a large constant (depending on $C$ and $r$) and add it to \eqref{key_est_minus}. Doing so, the terms in the {right-hand} side of \eqref{Car_res_cut} control the negative terms on the {right-hand} side of \eqref{key_est_minus}. When $R\gg r$, due to the weight $\Phi$, the terms resulting from cutting $w$ off to obtain $\bar{w}$ are absorbed by the positive terms in \eqref{key_est_minus}, giving \eqref{key_est}.

Once \eqref{key_est} is obtained, Theorem \ref{thm_cond_back} follows as in \cite{LP}, where the larger polynomial coefficients are used to absorb the {non-linearity} and the linear reaction term.

\bigskip  
 
The strategy of the proof Theorem \ref{thm_rates} ({Rate of convergence}) is to push {a} good estimate for the decay of $u$ from the smooth part of $M_t$ into the part of $M_t$ which is conical, and then into the part of $M_t$ where it behaves like the shrinker. The first two steps follow the strategy \cite{Wang_unique}, while the third requires a different argument. 

First, like in  \cite{Wang_unique}, we use the Carleman estimate of Escauriaza--Fer\'andez \cite{EF}, showing that around points in $M_0$ where $M_s$ is expressible as a graph with bounded $C^4$ norm over a ball of fixed radius for all $s\in [0,t]$, we have 
\begin{equation}
|u(x,t)|+|\nabla u(x,t)| \leq e^{-\delta/t},
\end{equation}
for some $\delta>0$ depending on the above bounds. {In particular, this} applies outside of $B(0,\varepsilon)$, where the flow $M_t$ is smooth up to time $0$. In fact, by using this estimate for the parabolic {dilations} of the flow, we get that outside of $B(0,R\sqrt{t})$ we have
\begin{equation}\label{decay_smooth}
|u(x,t)|+|\nabla u(x,t)| \leq \exp\left(-\frac{\delta|x|^2}{t}\right).
\end{equation}
Note that while for points $x$ with $|x| \geq \varepsilon$ the above decay is stronger than the one asserted in  Theorem \ref{thm_rates}, for point $x$ with $|x|=R\sqrt{t}$, \eqref{decay_smooth} gives no decay at all.

\bigskip

In order to push the decay further in, we use the fact, due to  Chodosh--Schulze, \cite{CS}, {that} there exists some $\varepsilon>0$ and $R<\infty$ such that $M_t$ is close to being conical in $B(0,\varepsilon)\backslash B(0,R\sqrt{t})$ for every $t$ sufficiently small. {Such} conical behaviour puts us in the position to use the methods of Escauriaza, Seregin and Sverak \cite{ESS} and Wang \cite{Wang_unique}. 

In \cite[Lemma 3.4]{Wang_unique}, for an asymptotically conical end of a {self-shrinker}, Wang used \textit{both} the self-shrinking behaviour of the end \textit{and} its conicality,  to {derive Carleman} estimates, depending on parameters $\delta$ which can be taken as \textit{any} value in the interval  $(0,1)$, and $\alpha \gg 1$. Wang's estimate, which holds for $M_t$, {a self-shrinking} AC flow, is of the form
\begin{align}\label{Wang_carl}
\int _0^{T}&\int_{M_t} \left(|\nabla_{M_t}v|^2+v^2\right) \mathcal{G} d\mu_tdt\\
&\leq \int _0^{T}\int_{M_t}\left(\frac{dv}{dt}+\Delta_{M_t}v\right)^2\mathcal{G} d\mu_tdt+\int_{M_T}|\nabla_{M_{t}}v|^2 \mathcal{G} d\mu_tdt,\nonumber
\end{align}
where $v$ {vanishes} at time $t=0$ and outside of the annulus $B(0,r)\backslash B(0,R)$, when $r\gg R$ and $R$ is sufficiently large, and where 
\begin{align}\label{final_intro}
            \mathcal{G}=\mathcal{G}_{\alpha,T}:=\mathrm{exp}\left(2\alpha(T-t)(|x|^{1+\delta}-R^{1+\delta})]+2|x|^2\right).
\end{align}
In Theorem \ref{Car_rough_con|} we show that \eqref{Wang_carl} holds true for any $M_t$ (not necessarily {self-shrinking}), provided the flow is \textit{roughly conical} in the annulus $\{x\;|\;R\leq  |x| \leq r\}$ (in a precise sense: see Definition \ref{rough_con_def}) , \textit{provided} $\delta$ is taken to be sufficiently close to $1$. While at some points a bit more care is needed in this analysis, this is a relatively minor modification of \cite{Wang_unique}. This is perhaps not that surprising, as in the original context of \cite{ESS}  only {conicality} was used, and as Kotschwar was able to amplify asymptotic {conicality} for Ricci flows to self-shrinking {behaviour}  \cite{Kot_cone}.

Estimate \eqref{Wang_carl} is used in \cite{Wang_unique} on carefully chosen cutoffs of the height $u$ of one {self-shrinking} flow over the other. The ``cutoff-ed" function is supported on the  annulus $B(0,r)\backslash B(0,R)$. After quite some work, combined with the decay rate \eqref{decay_smooth}, Wang gets that at the AC {self-shrinker} case, after taking $r\rightarrow \infty$ one has (roughly)
\begin{equation}
\int_{0}^{T}\int_{M_t-B(0,R)}u^2 \leq \exp(-\alpha),  
\end{equation}
which after taking the limit as $\alpha\rightarrow \infty$ gives the desired result. 

In our case of AC singularities, the {roughly} conical {behaviour} is {expected}, due to \cite{CS}, but only inside an annulus of shrinking inner {radius}{and} fixed outer {radius:} $B(0,r)\backslash B(0,R\sqrt{t})$.  {Rather} {than} taking the limit as $r\rightarrow \infty$ (which is impossible in our case since $M_t$ is not asymptotically conical), in Theorem \ref{con_cone} we optimize the choice of $\alpha$ in the estimate \eqref{Wang_carl} so that, coupling it with \eqref{decay_smooth} gives the best integral bound on the height square. This gives an estimate of the form
\begin{equation}
\int_0^t\int_{M_s\cap(B(0,r)-B(0,R\sqrt{t}))}u^2 \leq \exp\left(-\frac{1}{t^{(1-\delta)/2}}\right),
\end{equation}
which yields similar $C^1$ estimates outside $B(0,R\sqrt{t})$ in a standard fashion, after reducing the exponential ever so slightly. 

\bigskip

Finally, in Section \ref{sec_core} we push the decay into the region $\{x\in M_t \cap B(0,R\sqrt{t})\}$, which we call the core shrinker region. When considering the rescaled flows
\begin{equation}
N^i_{\tau}= \frac{M^i_{e^{\tau}}}{e^{\tau/2}},
\end{equation}
the $N^{i}_{\tau}$ converge in $B(0,2R)$ to a ball in the {self-shrinker} $\Sigma$, and the {height function} $w$ of one flow over the other solves a linear backwards parabolic equation. We then use a stability inequality from the theory of inverse problems \cite{Isakov}, which roughly states that a solution $w$ to a (backwards) parabolic equation on $\Omega\times [0,1]$ which has small Dirichlet \textit{and} Neumann values on some portion of the Dirichlet boundary, i.e. on  $\Gamma\times [0,1]$,  where $\Gamma\subseteq \partial \Omega$ is open, has to be respectively small inside the space time domain, without any assumption on the behaviour of $w$ at the initial or final time. Applying this result on coverings of the core shrinker region that have portions of the boundary outside of $B(0,R)$ allows for the final push of the decay of $u$ (which is related to $w$ by rescaling) into the core region, which then completes the proof of Theorem \ref{thm_rates}.  

\bigskip

{

As it is both easy and instructive for potential future applications towards the proof of Conjecture \ref{backwards_uniqueness_conjecture},  we end this introduction by showing how Theorem \ref{main_thm} implies the backwards uniqueness of low entropy flows in $\mathbb{R}^4$.

\begin{proof}[Proof of Corollary \ref{backwards_uniqueness_cor}]
Denote
\begin{equation}\label{t1def}
t_1=\inf \Big\{t\in [0,t_0] \; | \; M^1_s=M^2_s\; \textrm{for every} s\in [t,t_0]\Big\}
\end{equation}

By \cite{BW_cone}, each $(M^i_t)_{t\in [0,t_0]}$ encounters only finitely many singularities, and the regular set of $M^i_t$ is connected for every $t\in [0,t_0]$ (we conflate $M^i_t$  with its support). Thus 
\[
M^1_{t_1}=M^2_{t_1}
\]
Since $t_1$ is prior to the extinction time, \cite{BW_cone} further implies that either $t_1$ is a smooth time for $M_t$, or that $M_t$ encounters finitely many asymptotically conical  singularities at  time $t_1$.  In the former case, the smooth backwards uniqueness of \cite{KS_simple} yields a contradiction to \eqref{t1def}, unless $t_1=0$. In the latter case, our main theorem \ref{main_thm} (combined with Remark \ref{several_sings}) yields a contradiction to \eqref{t1def}. Thus $t_1=0$, concluding the proof.
\end{proof}
}

\subsection{Organization}
The organization of this paper is as follows. In Section 2, we recount the relevant theory from the literature, and in particular obtain the graphicality of one flow over the other. We further discuss that AC assumption and its validity. In Section 3 we prove Theorem \ref{thm_cond_back} (Conditional backwards uniqueness), demonstrating that a sufficiently fast convergence of the flows towards each other implies they coincide. In Section 4 we prove Theorem \ref{thm_rates} (Rate of convergence), demonstrating that the convergence rate required for Theorem \ref{thm_cond_back} (Conditional backwards uniqueness) is achieved. 

\subsection{Acknowledgements}
We are grateful to Brett Kotschwar, Arick Shao, Felix Schulze and Brian White for useful discussions. This project has received funding from the European Research
Council (ERC) under the European Union’s Horizon 2020 research and innovation programme, grant agreement No. 101116390. JMDH was supported by the
EPSRC through the grant EP/Y017862/1, Geometric Flows and the Dynamics of
Phase Transitions.

\section{{preliminaries}}\label{prelim_sec}
The purpose of this section is twofold: to introduce the necessary {theory from mean curvature flow literature required} in the rest of this paper to a {general mathematical audience}, and to reduce the geometric assumptions of Theorem \ref{main_thm} to analytic assumption. 
The {non-}geometric-analyst can safely skip all the proofs of this section and simply take their output as an assumption.

As was mentioned in the introduction, throughout the paper we will be concerned with compact backwards mean curvature flows: For $i=1,2$,  $(M^i_t)_{t\in [0,T]}\subseteq \mathbb{R}^{n+1}$ is compact for every $t\in[0,T]$, smooth outside of $(0,0)$ and satisfies
\begin{equation}
\left(\frac{d}{dt}x\right)^{\perp}=-\vec{H} \qquad (x,t)\neq (0,0).
\end{equation}
We have already defined what an asymptotically conical singularity is for {an} MCF. 

\begin{definition}[Asymptotically conical singularity]\label{AC_def_back}
We say that a BMCF $(M_t)_{t\in [0,T]} $has an asymptotically conical singularity at $(0,0)$ if the MCF $M_{T-t}$ has an AC singularity at $(0,T)$, according to definition \ref{AC_forwards}.
\end{definition}

That is, a BMCF $M_t$ encounters an AC singularity at $(0,0)$ if  there exists a smooth hypersurface $\Sigma\subseteq \mathbb{R}^{n+1}$ such that
\begin{equation}\label{tangent_simple}
\lim_{t\rightarrow 0}\frac{M_t}{\sqrt{t}}=\Sigma \qquad \textrm{in\;} C^{\infty}_{\mathrm{loc}}(\mathbb{R}^{n+1})
\end{equation}
and  there exists a {smooth} cone $\Lambda \subseteq \mathbb{R}^{n+1}$ {(i.e. $\Lambda -\{0\}$ is smooth)} such that
\begin{equation}
\lim_{\lambda\rightarrow 0} \lambda \Sigma = \Lambda \qquad \textrm{in\;} C^{\infty}_{\mathrm{loc}}(\mathbb{R}^{n+1}-\{0\}).
\end{equation}

As was already explained in the introduction, the above $\Sigma$ satisfies the elliptic partial differential equation 
\begin{equation}\label{srink_eq}
\vec{H}+\frac{x^{\perp}}{2}=0,
\end{equation}
where $x$ is the position vector and $x^{\perp}$ is the part of $x$  is normal to the tangent plane at $x$. Equivalently $(\sqrt{t}\Sigma)_{t\in (0,\infty)}$ is a BMCF.

\subsection{Analytic structure of coinciding conical flows}

In this section we collect some results regarding AC singularities, and the graphical gauge for our coinciding flows. All the results in this section either appear explicitly in (or as an implicit consequence of) the work of Chodosh-Schulze \cite{CS} and Wang \cite{Wang_unique}, or follow easily from combining their results with standard theory. 

our first theorem collects results from \cite{CS}. The proof given here is based on the same ideas as there, but is significantly shorter, as our definition assumes the uniqueness of tangents at an AC singularity, the main result of \cite{CS}.

\begin{theorem}[Structure of flows with AC singularity \cite{CS}]\label{AC_sing_struc}
Let $(M_t)_{t\in [0,T]}$ be a backwards MCF with an AC singularity at $(0,0)$, modeled on a {self-shrinker} $\Sigma$. Then 
\begin{enumerate}
\item (Type I singularity) There exists some $C<\infty$ such that 
\begin{equation}\label{A_inj_bd_cone}
\sup_{x\in M_t}|A(x,t)|\leq  \frac{C}{\sqrt{t}}, \qquad \mathrm{inj}^N_{M_t}  \geq \frac{\sqrt{t}}{C},
\end{equation}
where $\mathrm{inj}^N_{M_t}$ is the normal injectivity radius of $M_t$
\item (Conical curvature behaviour): For every $\varepsilon>0$ there exists $R<\infty$ such that 
\begin{equation}\label{A_inj_bd_away}
\sup_{x\in M_t\cap \{|x|\geq R\sqrt{t}\}}|A(x,t)|\leq  \frac{\varepsilon}{\sqrt{t}}
\end{equation}
\item\label{rough_con} (Rough {conicality}): There exist some  $C<\infty$ such that for every {$\varepsilon>0$} there exists some $T_0\in (0,T]$, $c<\infty$ and $r>0$ with the following significance:
For every $\lambda\geq 1$, the flow $\lambda M_{\lambda^{-2} t}$ for $t\in [0,\lambda^2 \tau]$ is $(C,c,\varepsilon)$ \textbf{roughly conical} in $B(0,\lambda r)$ (see Definition \ref{rough_con_def} below).

\item (Graphicality over the shrinker) There exists some $r>0$ and $T_0\in (0,T]$ such that $(M_t\cap B(0,r))_{t\in [0,\tau]}$ can be expressed as a normal graph of a smooth function $v$, with bounded $\mathrm{C}^1$, over a piece of the {self} shrinking flow $(\sqrt{t}\Sigma)_{t\in [0,T_0]}$. Moreover $v=o(1)$ as $t\rightarrow 0$.
\end{enumerate}
\end{theorem}

\begin{definition}\label{rough_con_def}
We say that a backwards  MCF $(M_t)_{t\in [0,T]}$ is $(C,c,\varepsilon)$ \textbf{roughly conical} at a ball $B(0,r)$ if for every $t\in [0,T]$ and $x\in B(0,r)-B(0,c\sqrt{t})$ we have
\begin{equation}\label{A_scales_like_cone}
|\nabla^i A(x)| \leq \frac{C}{|x|^{i+1}},\qquad i=0,\ldots 4 
\end{equation}
and 
\begin{equation}\label{x_is_like_cone}
|\langle x,\nu \rangle | \leq \varepsilon|x|.
\end{equation}
\end{definition}

\begin{proof}[Proof of Theorem \ref{AC_sing_struc}]
All the statements follow from the same construction: First, since $\Sigma$ is asymptotic to a cone, there exists a constant $C<\infty$ be such that for every $x\in \Sigma$ with $|x|\geq 1$
\[
|A(x)| \leq \frac{C}{|x|}, \qquad \mathrm{inj}^N(x) \geq \frac{|x|}{C}, 
\]
where $\mathrm{inj}^N$ is the normal injectivity radius of $\Sigma$ at $x$. Since $\sqrt{t}\Sigma$ is a BMCF, it follows from standard parabolic estimates (or the local maximum principle arguments of \cite{EH_interior}) that, after increasing $C$, if $x\in \Sigma$ and $|x|\geq 1$
\begin{equation}\label{A_der_bd_proof}
|\nabla^i A(x)| \leq \frac{C}{|x|^i},\qquad i=0,1,\ldots 4. 
\end{equation}

Since $\Sigma$ is asymptotic to a cone, it further follows that for every $\varepsilon>0$ with $\varepsilon<C^{-1}$, there exists some $c=c(\varepsilon) \gg \varepsilon^{-4}$ such that 
\begin{equation}\label{almost_cone}
\left|\langle x, \nu \right \rangle| \leq \varepsilon^2 |x|, \qquad x\in \Sigma,\; |x|\geq c,
\end{equation}
and such that for every $x\in \Sigma$ with $|x|\geq c$ there exists a linear hyperplane $P$ through $x$ (and the origin), such that $\Sigma\cap B(x,\varepsilon c)$ is a Lipschitz graph over $P$, with Lipschitz constant bounded by $\varepsilon^2$.

\bigskip

Since $M_t/\sqrt{t}\rightarrow \Sigma$, we can take $T_0$ such that that for every $t\in [0,T_0]$,  the hypersurface $M_{t}/\sqrt{t}$ can be written as a normal graph of a smooth function $u:B(0,100c)\cap \Sigma \rightarrow \mathbb{R}$ with $\mathrm{C}^{10}$ norm bounded by $\varepsilon^2$. Take  $r=99c\sqrt{T_0}$. Let us further denote $\Sigma_s=\sqrt{s}\Sigma$, and 
\[
\mathcal{A}_t=B(0,99c\sqrt{t})-B(0,2c\sqrt{t})
\]

Applying  pseudolocality \cite[Theorem 1.5]{INS} and the Ecker-Huisken gradient estimate  \cite{EH_interior}, together with the choices of constant above we get that:

\begin{starenv*}
\starlabel{important}
for every $t\in (0,T_0]$ and every point $p\in \Sigma_t\cap \mathcal{A}_t$, there exists some linear hyperplane $P$ through $p$  (and the origin) such that both $(M_s)_{s\in [0,t]}$ and $(\Sigma_s)_{s\in [0,t]}$ are graphical, with $C^1$ norm bounded by $\varepsilon/10$, over the ball $B(p,\frac{\varepsilon c \sqrt{t}}{2})$ in $P$.
\end{starenv*}
In light of the choice of $T_0$ and \eqref{A_der_bd_proof}, we have 
\begin{align}\label{A_der_bd_all_proof}
&|\nabla^i A(x,s)| \leq \frac{C}{|x|^{i+1}}, \qquad x\in M_s\cap \mathcal{A}_t,\; s\in [0,t],\; i=0,\ldots 4\nonumber \\
&|\nabla^i A(x,s)| \leq \frac{C}{|x|^{i+1}}, \qquad x\in \Sigma_s\cap \mathcal{A}_t,\; s\in [0,t], \; i=0,\ldots 4.
\end{align} 
Finally, observe that if $t\in [0,T_0]$ and $x\in M_t$ then either:
\begin{enumerate}[(i)]
\item $x\notin B(0,r)$ or,
\item there exists some $t_1\in [t,T_0]$ such that $x\in A_{t_1}$, or 
\item $x\in B(0,2c\sqrt{t})$.
\end{enumerate} 
Having established the above structure, (1)-(4) follow readily.

Indeed, for (1) and (2), it suffices to prove it for $t\in (0,T_0]$. Taking and $x\in M_t$, if option (iii) above happens then the curvature is bounded, as outside an open ball is compact, in space time, and smooth. If (ii) happens then estimate \eqref{A_der_bd_all_proof} and the definition of $\mathcal{A}_{t_1}$ and the choice of constants imply 
\begin{equation}
|A(x)| \leq \frac{C}{|x|} \leq \frac{C}{2c\sqrt{t_1}}  \leq \frac{\varepsilon}{\sqrt{t_1}} \leq \frac{\varepsilon}{\sqrt{t}}.
\end{equation}   
Finally, if (iii) holds then $|A(t)| \leq \frac{2D+2}{\sqrt{t}}$, where $D$ is a bound on the second fundamental form of $\Sigma$. The same argument also provides the asserted normal injectivity radius bound. This completes the proof of both (1) and (2). 

Since the curvature estimates in the region $\mathcal{A}_{t}$ are scale invariant, this also proves \eqref{A_scales_like_cone}, while \eqref{x_is_like_cone} follows as a consequence of \starref{important}: for every $t_1\in [t ,T_0]$ a portion of $M_t\cap \mathcal{A}_{t_1}$ can be represented in a ball as a graph over a hyperplane through the origin with $C^1$ norm bounded by $\varepsilon/10$. This establishes (3), but also (4):  As the same  planar graphical representation holds true for $\Sigma_t$ in $\mathcal{A}_{t_1}$, this also shows the graphical representation of $M_t$ over $\Sigma_t$ in  $B(0,r)-B(0,2c\sqrt{t})$. The graphical representation inside $B(0,2c\sqrt{t})$ follows from the choice of $T_0$, while the smoothness of the graphical function follows since it satisfies a parabolic PDE with smooth coefficients (see Lemma \ref{graph_eq_evolve} below) . This completes the proof of the theorem.   
\end{proof}

At the end of the above proof, we have used that the height function of one BMCF over another satisfies a quasilinear parabolic PDE. This is standard, but as this would be used repeatedly throughout this paper, we will include it here.

\begin{lemma}[\cite{Her_reif}, Lemma 4.18]\label{graph_eq_evolve}
Let $(M^i_{t})_{t \in (0,t_0]}$  $i=1,2$ be two  mean curvature flows such that $M^2_t$ is a normal graph of a function $u$ over $M^1_t$, for which both $|A(x)||u(x)|$ and $|\nabla u(x)|$ are sufficiently small. Then
\begin{align}\label{par_eq}
u_{t}&=-\Delta u - |A|^2u+(1+\varepsilon_1(x,u,\nabla u))\mathrm{div}(\mathcal{E} (\nabla u))\\
&+\varepsilon_1(x,u,\nabla u)(\Delta u+\mathrm{div}(\mathcal{E} (\nabla u)))+\varepsilon_2(x,u,\nabla u),
\end{align}
where $\mathcal{E}:TM^1_{t}\rightarrow TM^1_{t}$ and $\varepsilon_i:M^1_{t}\rightarrow \mathbb{R}$ are of the form
\begin{align}\label{E_est}
\mathcal{E}=A\ast u+\nabla u, \qquad  \varepsilon_1= A \ast u +\nabla u, \qquad \varepsilon_2 = (A\ast u +\nabla u) \ast  (A\ast u +\nabla u).
\end{align}    
where the above $\ast$ notation indicates that the expression on the left hand side (resp\.  its derivatives), are bounded by constant times the corresponding norm of the {right-hand} side (resp\. its derivatives).
\end{lemma}
An analogous theorem holds for the rescaled MCF \cite[Lemma A.1]{DHH}

\bigskip

The second main theorem of this section is concerned with two BMCFs with AC singularities at $(0,0)$, which coincide at time $0$.

\begin{theorem}[Global graphicality]\label{graph_over_one_another}
Let $(M^i_t)_{t\in [0,T]}$ for $i=1,2$ be two BMCF in $\mathbb{R}^{n+1}$ with $M^1_0=M^2_0$, such that the $M^i_t$ have {an} AC {singularity} at $(0,0)$. Then there exists $T_0\in (0,T]$ such that for every $t\in [0,T_0]$, and {away from} the space time point $(0,0)$, $M^2_t$ can be written as a normal graph of a smooth function $u:M^1_t\rightarrow \mathbb{R}$ satisfying
{\begin{align}
|u(x,t)| = o(\sqrt{t}),\\
|\nabla u(x,t)|=o(1).
\end{align} }
\end{theorem}
\begin{proof}
By \cite{Wang_unique} (see Lemma \ref{same_shrink} below), the singularity of both $M^1_t$ and $M^2_t$ are modelled on the same AC shrinker $\Sigma$. Let $\varepsilon$ and let $T_0, C,c,R$ and $r$ be such that the conclusions of Theorem \ref{AC_sing_struc} holds for both of the  $M^i_t$. Assume without loss of generality that $\sqrt{n}R<c$  and that, as in the above proof,  for every $t\in [0,T_0]$ the hypersurfaces $M^i_{t}/\sqrt{t}$ can be written as a normal graph of a smooth function $v^i:B(0,100c)\cap \Sigma \rightarrow \mathbb{R}$ with $\mathrm{C}^{10}$ norm bounded by $\varepsilon^2$. This, together with \eqref{A_inj_bd_cone} shows the existence of the asserted $u$ in $B(0,50c\sqrt{t})$, with bounds
\begin{equation}\label{u_ins_shr}
|u(x,t)| \leq  \varepsilon\sqrt{t}, \qquad |\nabla u(x,t)|\leq \varepsilon,\qquad x\in B(0,50c\sqrt{t})
\end{equation}
Given any $x_0\in M^1_0=M^2_0$ with $x_0\neq 0$, denote by $(x^i(t))_{t\in [0,T_0]}$ its worldline under the BMCF $M^i_t$ (that is, the integral curve for the ODE $\sd{t} x(t)= H_{M^i_t}(x)\nu_{M^i_t}(x), x(0)=x_0$). Setting 
\[
t_1=\min\{\inf\{t\in [0,T_0]\;:\;x^1(t)\in B(0,20c\sqrt{t})\},T_0\}
\]
we conclude that $t_1>0$. Integrating \eqref{A_inj_bd_away} and noting that  $\sqrt{n}R<c$, we see that $|x^2(t)| \geq R\sqrt{t}$ for every $t\in [0,t_1]$. Moreover,  we see that for every $t\in [0,t_1]$
\[
|x^1(t)-x^2(t)|\leq 2\sqrt{n}\varepsilon\sqrt{t}.
\]
Combined with \eqref{u_ins_shr} this implies that for every $t\in [0,T_0]$
\begin{equation}
d_{\mathrm{Hauss}}(M^1_t,M^2_t) \leq 2\sqrt{n}\varepsilon\sqrt{t}.
\end{equation}
Together with the normal injectivity radius bound of \eqref{A_inj_bd_cone} (for both $i=1,2$), this implies that $M^2_t$ is a normal graph over $M^1_t$ of a function $u$ satisfying
\begin{equation}
|u(x,t)|  \leq 2\sqrt{n}\varepsilon \sqrt{t}.
\end{equation}
The $C^1$ bound follows {from} Theorem \ref{AC_sing_struc} in $B(0,r)$, and from the flows being smooth elsewhere.
\end{proof}

We have used the following Lemma, which follow easily from \cite{Wang_unique}

\begin{lemma}\label{same_shrink}
Suppose $(M^i_t)_{t\in [0,T]}$ for $i=1,2$ are two backwards mean curvature flows in $\mathbb{R}^{n+1}$ encountering an AC {singularity} at $(0,0)$, such that $M^i_t/\sqrt{t}\rightarrow \Sigma_i$. If  $M^1_0=M^2_0$. Then $\Sigma_1=\Sigma_2$.
\end{lemma}
\begin{proof}
By Theorem \ref{AC_sing_struc} (4) (which is again from \cite{CS})
\begin{equation}
C_i=\lim_{\lambda \rightarrow 0 } \lambda {\Sigma_i}=\lim_{\lambda \rightarrow \infty} \lambda M^i_0.
\end{equation}
Thus $C_1=C_2$, which by \cite{Wang_unique} implies that $\Sigma_1=\Sigma_2$.  
\end{proof}

{W}e record the following evolution equations under BMCF.
\begin{lemma}[\cite{Huisken}]\label{eqnofmotion}
	Given a  BMCF, $(M_t)_{t\in[0,T]}$, the following evolution equations hold
    \begin{align*}
        \left(\pd{}{t}+\Delta_{M_t}\right)H&=-\norm{A}^2 H,\\
        \Delta_{M_t} \mathbf{n}&=\nabla H -|A|^2 \mathbf{n},\\
        \Delta_{M_t}\ip{x}{\mathbf{n}},
        &=H+\ip{x}{\nabla H}-|A|^2\ip{x}{\mathbf{n}},\\
        \sd{t} \ip{x}{\mathbf{n}}&=H-\ip{x}{\nabla H}.
    \end{align*}
\end{lemma}
\subsection{Validity of the AC singularity assumption}

The {assumption} of the main Theorem of this paper, Theorem \ref{main_thm}, requires, in addition to $M^1_0=M^2_0$, that both $M^1_t$ and $M^2_t$  have an AC singularity, as in Definition \ref{AC_def_back} at $(0,0)$. This section serves to show that this assumption is {natural} in some context: there are situations in which such singularities are expected, and moreover, in {these} situations, $M^2_t$ has an AC singularity if and only if $M^1_t$ has an AC singularity. That is to say, the AC singularity {may} be taken as an assumption for only one of the flows.  The result in this section will not be used elsewhere in this paper.

We start by introducing some standard terminology: Denote by $\mathcal{M}^i$ the space-time track of $M^i_t$, namely
\begin{equation}
\mathcal{M}^i=\{(x,t)\;|x\in M_t\}.
\end{equation}
We will also conflate $\mathcal{M}^i$ with the associated multiplicity one (backwards) Brakke flow (see \cites{Brakke,Ilm_elip}). For every $\lambda>0$ let $\mathcal{D}_{\lambda}\mathcal{M}^i$ be the (backwards) parabolic {dilation} of $\mathcal{M}^i$, namely
\[
\mathcal{D}_{\lambda} \mathcal{M}^i=\{(\lambda x, \lambda^2 t)\;|\;(x,t)\in \mathcal{M}^i\}.
\]
It follows from Huisken's monotonicity formula \cites{Huisken_mono, Ilm_surfaces} that any sequence $\lambda_i\rightarrow \infty$ has a sub-sequence {(reindexed)} along which
\begin{equation}\label{tang_flow}
\mathcal{D}_{\lambda_i}\mathcal{M}^1\rightarrow \tilde{N},
\end{equation}
where $\mathcal{N}$ is a self shrinking (backwards) Brakke flow  whose $t=1$ time-slice is {a} varifold $\Sigma$ satisfying \eqref{srink_eq} in the varifold sense. Any $\mathcal{N}$ obtained as a subsequential limit as in \eqref{tang_flow}, which is taken in {the sense} of (backwards) Brakke flows, is called a \textbf{tangent flow} to $M^1_t$ at $(0,0)$.

The self-similarity of $\mathcal{N}$ also implies \cite{Ilm_lec} (see also \cite{Wang_asym}) that the set of points $x\in \mathbb{R}^{n+1}$ which the flow $\mathcal{N}$ reaches at time $0$ forms a cone $C$ in $\mathbb{R}^{n+1}$, in the sense that
\begin{equation}
x\in C \Rightarrow \lambda x\in C\;\textrm{for every } \lambda>0. 
\end{equation}

\bigskip

Let $\mathcal{N}$ be a tangent flow to $M^1_t$ at $(0,0)$. We make the following two assumptions, which we {justify soon after}:

\begin{assumption}\label{nat_ass1}
The time $t=1$ of $\mathcal{N}$,  $\Sigma$, is smooth, and the convergence in \eqref{tang_flow} is with multiplicity one.
\end{assumption}

\begin{assumption}\label{nat_ass2}
The cone $C$ of points that are reached by $\mathcal{N}$ at time $0$ is a smooth hypersurface away from $0$, and 
\[
\lambda \Sigma \rightarrow C \; \textrm{smoothly in $\mathbb{R}^{n+1}-\{0\}$ as $\lambda\rightarrow 0$}
\]
\end{assumption}

Assumption \ref{nat_ass1} is essentially the only case of singularities of MCF which one can currently  handle in general\footnote{The smoothness of $\Sigma$ can be relaxed in some cases to the assumption that the singular set of $\Sigma$ has $(n-1)$-dimensional Hausdorff measure zero}. By \cite{Bamler_Kleiner}, Assumption \ref{nat_ass1} holds true when $n=2$. 
In this very case it's a theorem of Wang \cite{Wang_asym} that Assumption \ref{nat_ass2} is valid for any singularity for BMCF in $\mathbb{R}^3$, unless $\mathcal{N}$ has a cylindrical end, which conjecturally can only happen if $\Sigma$ itself is, up to rotation, the cylinder $\mathbb{S}^1(\sqrt{2})\times \mathbb{R}$ \cite{Ilmanen_problems}. Thus, when $n=2$ presumably either \eqref{nat_ass2} holds or the singularity at $(0,0)$ is a neck singularity. For higher dimensions \eqref{nat_ass2} is still natural and have been the subject of much study. For instance, if $M^1_t$ has entropy\footnote{
Recall that the Gaussian entropy of a {hypersurface} $N\subseteq \mathbb{R}^{n+1}$ and of the flow $M^i_t$, defined as
\begin{equation}
\mathrm{Ent}[N]=\sup_{\substack{x_0 \in \mathbb{R}^{n+1} \\ \lambda>0}}\frac{1}{(4\pi)^{n/2}}\int_{\lambda(M-x_0)}\exp\left(-\frac{|x|^2}{4}\right), \qquad \mathrm{Ent}[\mathcal{M}^i]=\mathrm{Ent}[M^i_T]
\end{equation}} less than that of the neck cylinder, then both Assumption  \ref{nat_ass1} and Assumption\ref{nat_ass2} hold \cite{BW_cone}.

\bigskip 

Given Assumptions \ref{nat_ass1} and \ref{nat_ass2},  a theorem of Chodosh and Schulze \cite{CS} states that $\mathcal{N}$ is the unique tangent flow to $\mathcal{M}^1$ at $(0,0)$. In particular, this implies that $M^1_t$ encounters an AC singularity at $(0,0)$, according to Definition \ref{AC_def_back}

\bigskip 

The following lemma  serves to show that the assumption that $M^2_t$ experiences an AC singularity follows from $M^1_t$ encountering an AC singularity, in the cases where an AC singularity could be expected to begin with.
\begin{lemma}\label{AC_from_other}
Suppose $(M^1_t)_{t\in [0,T]}$ and $(M^2_t)_{t\in [0,T]}$ are two backwards mean curvature flow in $\mathbb{R}^{n+1}$, smooth outside of $(0,0)$,
such that $M^1_0=M^2_0$ and such that $M^1_t$ has an asymptotically conical singularity at $(0,0)$, modelled on a {self-shrinker} $\Sigma$.  Assume either 
\begin{enumerate}
\item $n=2$ or,
\item $n=3$ and the entropy of $M^2_t$ is at most the entropy of $\mathbb{S}^2\times \mathbb{R}$.
\end{enumerate}
Then $M^2_t$ has an asymptotically conical singularity at $(0,0)$.
\end{lemma}

\begin{proof}
Let $\mathcal{N}^2=\{N_t\}_{t\in [0,\infty)}$ be a tangent flow to $M^2_t$, obtained from the sequence of dilations $\mathcal{D}\lambda_j\mathcal{M}^2$ with $\lambda_j\rightarrow \infty$. Denote $\Sigma^2=N_{1}$.   

When $n=3$, it follows from the entropy assumption and \cite{BW_cone} that $\Sigma^2$ is smooth, multiplicity one, and either asymptotically conical or compact. Since $M^2_0=M^1_0$ doesn't consist of a single point, {it} follows that $\Sigma^2$ is asymptotically conical. By \cite{CS} $M^2_t$ has an AC singularity at $(0,0)$, in accordance with Definition \ref{AC_def_back}.

When $n=2$, by \cites{Bamler_Kleiner,Ilm_surfaces} $\Sigma^2$ is smooth and the convergence is multiplicity one. By \cite{Wang_asym}, if $\Sigma^2$ is not AC, then it has some cylindrical end, but then $N_0$ must contain a ray of singularities. This is impossible, as 
\[
N_0=\lim_{j\rightarrow \infty}M^2_0=\lim_{j\rightarrow \infty}M^1_0,    
\]
and the {right-hand} side is a smooth cone away from $0$.
\end{proof}

\section{A conditional backwards uniqueness theorem}

In this section, we will prove Theorem \ref{thm_cond_back} (Conditional backwards uniqueness), which gives an upper bound of convergence rates of two singular BMCFs to one another.

\bigskip

Theorem \ref{thm_cond_back} is a corollary of the following more general (conditional) backwards uniqueness for solution to backwards heat equation along a BMCF with an AC singularity. Given a BMCF $M_t$ on $[0,T]$, smooth outside of $(0,0)$, denote by  $L$ be the backwards heat operator
\begin{equation}
L=\partial_t+\Delta_{g(t)}.
\end{equation} 

\begin{theorem}[Backwards uniqueness for quasilinear heat-type equations assuming fast convergence]\label{main_technical_uniqueness_thm}
Let $(M_t)_{t\in [0,T]}$ be a BMCF in $\mathbb{R}^{n+1}$, which is smooth outside of $(0,0)$ with an AC singularity at $(0,0)$.  Suppose that  $u:[0,T]\times M\rightarrow \mathbb{R}$ is a solution to 
\begin{equation}
Lu=G(x, u,\nabla u)
\end{equation}
where $G$ satisfies the pointwise estimate
\begin{equation}\label{G_est}
|G(x,u,\nabla u)|^2 \leq \frac{Ku^2}{t^2}+\frac{o(1)}{t}|\nabla u|^2,
\end{equation}
where $K<\infty$. Suppose further that there exists some $\beta_1>0$ such that $u$ satisfies
 \begin{equation}\label{u_asymptotics_main}
\left|u\exp\left(\frac{1}{t^{\beta_1}}\right)\right|+\left|\nabla u\exp\left(\frac{1}{t^{\beta_1}}\right)\right| \leq D<\infty.
\end{equation}
Then $u\equiv 0$. 
\end{theorem}
 
The proof of Theorem \ref{main_technical_uniqueness_thm} relies on two Carleman estimates inspired by the classical one of Lees and Protter \cite{LP}, which was used  to prove a theorem similar to Theorem  \ref{main_technical_uniqueness_thm} for uniformly parabolic operators on a domain, where $G$ satisfies the bounds 
\[
|G(u)|^2 \leq K(|u|^2+|\nabla u|^2).
\] 
On a technical note, dealing with the Type I singularity, as well as with the non-uniform parabolicity at the singular time necessitates using a different, more singular weights, and also a growth assumption (to be validated in the subsequent section of this paper). 

More importantly, in order for any $\beta_1>0$ to suffice in \eqref{u_asymptotics_main}, one needs to 
deal with the regions away from and within the core shrinker scale differently, and patch these Carleman estimates together. This requires utilizing the smooth behaviour and shrinking behaviour - not only on the conical part of the {self-shrinker}. Doing so is probably one of the main technical novelties of this paper.  
\bigskip

To set up the Carleman estimate, given parameters $\alpha,\beta>0$, consider the weight
\[
\mathcal{G}=\mathcal{G}_{\alpha,\beta}:=e^{\frac{\alpha}{t^\beta}}.
\]
and consider the weighted $L^2$ and $W^{1,2}$ space of functions $u:M_t\rightarrow \mathbb{R}$, whose (pseudo) norms are given by
\begin{equation}
\|u\|^2_{\mathcal{G}}=\int_{0}^T\int_{M_t}u^2\mathcal{G},\qquad \|u\|^2_{\mathcal{G},1}=\int_{0}^T\int_{M_t}|\nabla u|^2\mathcal{G}.
\end{equation}

The main technical estimate of this section is the following Carleman type estimate: 

\begin{theorem}[Smooth global Carleman]\label{smooth_global_Carleman}
Let $(M_t)_{t\in [0,T_1]}$ be a BMCF in $\mathbb{R}^{n+1}$ which is smooth outside of $(0,0)$ with an AC singularity at $(0,0)$. Given $\beta_1>0$ and $
\beta\in (0,\beta_1)$ there exist $T_0\in [0,T_1]$, $\varepsilon>0$  and $\alpha_0<\infty$ with the following significance:

Suppose $T\leq T_0$ and $u:(M_t)_{t\in [0,T]}\rightarrow \mathbb{R}$ is a smooth function such that $u(\cdot,T)=0$ and $u(\cdot ,0)=0$, and suppose that 
\begin{equation}\label{u_asymptotics}
\left|u\exp\left(\frac{1}{t^{\beta_1}}\right)\right|+\left|\nabla u\exp\left(\frac{1}{t^{\beta_1}}\right)\right| \leq D<\infty.
\end{equation}
Then for  $\alpha\geq \alpha_0$
\[
\|Lu\|^2_{\mathcal{G}} \geq  \|t^{-1-\beta/2}u\|^2_{\mathcal{G}}+\varepsilon \|t^{-1/2}u\|^2_{\mathcal{G},1}.
\]
\end{theorem}

As was mentioned above, Theorem \ref{smooth_global_Carleman} is actually a combination of two Carleman estimates, the first of which is similar to Theorem \ref{smooth_global_Carleman}, but with some negative terms in a rescaled ball.  

To facilitate this estimate, given $r<\infty$ we denote the interior region and exterior region by
\[
I^r_t=\{x\in M_t\;|\; |x|\leq r\sqrt{t}\} \qquad E^r_t=M_t-I^r_t.
\]
and denote,
\begin{equation}
\|u\|^2_{E^{r}_t,\mathcal{G}}=\int_{0}^T\int_{E^r_t}u^2\mathcal{G},\qquad \|u\|^2_{E^r_t,\mathcal{G},1}=\int_{0}^T\int_{E^r_t}|\nabla u|^2\mathcal{G},
\end{equation}
and similarly for the interior region
\begin{equation}
\|u\|^2_{I^{r}_t,\mathcal{G}}=\int_{0}^T\int_{I^r_t}u^2\mathcal{G},\qquad \|u\|^2_{I^r_t,\mathcal{G},1}=\int_{0}^T\int_{I^r_t}|\nabla u|^2\mathcal{G}.
\end{equation}

\begin{lemma}\label{first_lem}
Let $M_t$ be a BMCF on $[0,T]$, smooth outside of $(0,0)$, with AC singularity at $(0,0)$.
Suppose  $u:M\times [0,T]\rightarrow \mathbb{R}$ is a  smooth function such that $u(\cdot,T)=0$ and $u(\cdot ,0)=0$, and suppose that there exists some $\beta_1>0$ such that  
\begin{equation}
\left|u\exp\left(\frac{1}{t^{\beta_1}}\right)\right|+\left|\nabla u\exp\left(\frac{1}{t^{\beta_1}}\right)\right| \leq D<\infty.
\end{equation}
There exists some $C<\infty$ such that for every $\beta\in (0,\beta_1)$ and $\varepsilon>0$ there exists some $r>0$ such that for every $\alpha \geq 1$ 
\begin{align}
\|Lu\|^2_{\mathcal{G}} &\geq \frac{\alpha \beta(1+\beta/2)}{2} \|t^{-1-\beta/2} u\|^2_{E^r_t,\mathcal{G}}-\varepsilon\|t^{-1/2}u\|^2_{E^r_t,\mathcal{G},1}\\
&-C\alpha\beta\|t^{-1-\beta/2}u\|^2_{I^r_t,\mathcal{G}}-C\|t^{-1/2}u\|^2_{I^r_t,\mathcal{G},1}
\end{align}
\end{lemma}

\begin{proof}
Setting 
\[
z:=\sqrt{\mathcal{G}}u=e^{\frac{\alpha}{2t^{\beta}}}u,
\]
we easily compute 
\[
Lu = \mathcal{G}^{-1/2}\frac{\partial z}{\partial t}+\mathcal{G}^{-1/2}\Delta_{g(t)}z+\frac{\alpha \beta}{2t^{1+\beta}}\mathcal{G}^{-1/2}z.
\] 
Thus  
\begin{equation}\label{to_estimate_glob_car_lemma}
\| Lu\|^2_{\mathcal{G}}\geq  2\left\langle \frac{\partial z}{\partial t}, \Delta_{g(t)}z \right\rangle + \alpha\beta \left\langle \frac{\partial z}{\partial t}, \frac{z}{t^{1+\beta}} \right\rangle,
\end{equation}
provided the right-hand side is defined. Note that the inner product appearing {here} is the standard, unweighted inner product in space time
\[
\langle z_1, z_2 \rangle = \int_0^T\int_{M_t}z_1z_2d\mathrm{Vol}_tdt.
\]
We will integrate by parts the two terms on the right-hand side  of \eqref{to_estimate_glob_car_lemma} to obtain expressions which do not involve a time derivative of $z$.

For the second term on the right-hand side, observe 
\begin{align}\label{to_estimate_glob_car_lemma_1}
 \alpha\beta\left\langle \frac{\partial z}{\partial t}, \frac{z}{t^{1+\beta}} \right\rangle&= \frac{\alpha \beta}{2} \int\partial_t \int \left(\frac{z^2}{t^{1+\beta}}\right)+ \frac{\alpha \beta(1+\beta)}{2}\left\|\frac{z}{t^{1+\beta/2}}\right\|^2- \frac{\alpha \beta}{2}\left\|H\frac{z}{t^{(1+\beta)/2}}\right\|^2 \nonumber \\
 &= \frac{\alpha \beta(1+\beta)}{2}\left\|\frac{z}{t^{1+\beta/2}}\right\|^2- \frac{\alpha \beta}{2}\left\|H\frac{z}{t^{(1+\beta)/2}}\right\|^2,
\end{align}
where the right-hand side is well defined by \eqref{u_asymptotics}, and the second equality is valid since the temporal boundary values of $z^2/t^{1+\beta}$ are zero since $u(\cdot,T)=0$ and by \eqref{u_asymptotics} again. Note that the appearance of the involving $H$ is due to the time derivative of the volume form
\begin{equation}
\frac{d}{dt}\mathrm{Vol}_t=H^2\mathrm{Vol}_t.
\end{equation}

For the first term of the right-hand side of \eqref{to_estimate_glob_car_lemma}, observe that,
\begin{equation}\label{first_term_bla}
2\frac{\partial z}{\partial t}\Delta_{g(t)}z=2\mathrm{div}_{g(t)}\left(\frac{\partial z}{\partial t}\nabla_{g(t)}z\right)-2\left\langle\nabla_{g(t)}{\frac{\partial z}{\partial t}},\nabla_{g(t)}z\right\rangle_{g(t)},
\end{equation}
{ where $\langle \cdot, \cdot \rangle_{g(t)}$ denotes the pointwise inner product (Riemannian metric) on $M_t$}
Now,
\begin{equation}
2\left\langle\nabla_{g(t)}{\frac{\partial z}{\partial t}},\nabla_{g(t)}z\right\rangle_{g(t)} = 2g^{ij}\frac{\partial^2z}{\partial x_i \partial_t}\frac{\partial z}{\partial x_j}=\frac{d}{dt}\left(g^{ij}\frac{\partial z}{\partial x_i}\frac{\partial z}{\partial x_j}\right)- \frac{dg^{ij}}{dt}\frac{\partial z}{\partial x_i}\frac{\partial z}{\partial x_j},
\end{equation}
which, after taking  into account the evolution of the metric, 
\begin{equation}
g'(t)=2HA,\qquad
\end{equation}
yields 
\begin{equation}
2\left\langle\nabla_{g(t)}{\frac{\partial z}{\partial t}},\nabla_{g(t)}z\right\rangle_{g(t)} = 2g^{ij}\frac{\partial^2z}{\partial x_i \partial_t}\frac{\partial z}{\partial x_j}=\frac{d}{dt}g(\nabla z,\nabla z)+2HA(\nabla z,\nabla z).
\end{equation}
Substituting this into \eqref{first_term_bla}, the  divergence theorem, the evolution of the volume form,  and the boundary conditions once more we have
\begin{equation}
2\left\langle \frac{\partial z}{\partial t}, \Delta_{g(t)}z \right\rangle =  -\iint 2HA(\nabla_{g(t)}z,\nabla_{g(t)}z)+\iint H^2|\nabla z|^2.
\end{equation}

Note that since $M_t$ has {an AC} singularity at $0$, it follows from Theorem \ref{AC_sing_struc} (Structure of flows with AC singularity) that there exists some $C$ such that 
\begin{equation}\label{bounds_always}
H^2g-2HA \geq -\frac{Cg}{t},\qquad H^2 \leq \frac{C}{t},
\end{equation}
and there exists some $r<\infty$ such that for every $x\in E^r_t$ 
\begin{equation}
H^2g-2HA \geq -\frac{\varepsilon g}{t},\qquad H^2 \leq \frac{\beta}{4t} 
\end{equation}

Combining these with the definition of $z$ and  \eqref{to_estimate_glob_car_lemma}, \eqref{to_estimate_glob_car_lemma_1} yields the desired estimate.
\end{proof}

\begin{theorem}[Global Carleman with interior errors]\label{smooth_global_Carleman_neg_term}
Under the same conditions and notations of  Lemma \ref{first_lem},  there exist $T_0>0$, $\delta>0$  and $\alpha_0<\infty$ such that if  $\alpha\geq \alpha_0$, $T\leq T_0$ 
\begin{align}
\|Lu\|^2_{\mathcal{G}} &\geq  \alpha \beta \delta\|t^{-1-\beta/2}u\|^2_{E^r_t,\mathcal{G}}+ \|t^{-1/2}u\|^2_{E^r_t,\mathcal{G},1}\\
& -C\alpha\beta\|t^{-1-\beta/2}u\|^2_{I^r_t,\mathcal{G}}-C\|t^{-1/2}u\|^2_{I^r_t,\mathcal{G},1}.
\end{align}
\end{theorem}

\begin{proof}
We start by writing
\begin{equation}
\left\langle \frac{\sqrt{\mathcal{G}}}{t^{1+\beta/2}}u, \sqrt{\mathcal{G}}t^{\beta/2}Lu \right\rangle= \left\langle \frac{1}{t}e^{\frac{\alpha}{t^{\beta}}}u, \frac{\partial u}{\partial t}\right\rangle +\left\langle \frac{1}{t}e^{\frac{\alpha}{t^{\beta}}}u, \Delta_{g(t)}u\right\rangle.
\end{equation}
As before, using the boundary conditions and \eqref{bounds_always} we have
\[
\left\langle \frac{1}{t}e^{\frac{\alpha}{t^{\beta}}}u, \frac{\partial u}{\partial t}\right\rangle{\geq}\frac{\alpha\beta }{2}\iint  \frac{\mathcal{G}}{t^{2+\beta}} u^2+\frac{1-C}{2}\iint  \frac{\mathcal{G}}{t^{2}} u^2,
\] 
so for every $\delta>0$, we can take $T_0$ sufficiently small so that
\begin{equation}\label{global_car_tech_est1}
\left|\left\langle \frac{1}{t}e^{\frac{\alpha}{t^{\beta}}}u, \frac{\partial u}{\partial t}\right\rangle\right| \leq  \frac{1+\delta}{2} \alpha \beta \left\|\frac{u}{t^{1+\beta/2}}\right\|^2_{\mathcal{G}},
\end{equation}
{provided} $T$ is sufficiently small. Using the divergence theorem on every time slice, we also get
\begin{align}\label{global_car_tech_est_2}
\left\langle\frac{1}{t}\mathcal{G}u, \Delta_{g(t)}u\right\rangle=\iint\frac{1}{t}\mathrm{div}_{g(t)}\left(\mathcal{G} u\nabla u\right)-\iint\frac{1}{t}\mathcal{G}|\nabla u|^2=-\|t^{-1/2}u\|_{\mathcal{G},1}^2.
\end{align}

\bigskip

Combining \eqref{global_car_tech_est1} and \eqref{global_car_tech_est_2} we get
\begin{align}
\|t^{-1/2}u\|_{E^r_t,\mathcal{G},1}^2  \leq \|t^{-1/2}u\|_{\mathcal{G},1}^2 &\leq \left|\left\langle \frac{\sqrt{\mathcal{G}}}{t^{1+\beta/2}}u, \sqrt{\mathcal{G}}t^{\beta/2}Lu \right\rangle\right|+ \frac{(1+\delta)\alpha\beta}{2} \left\|\frac{u}{t^{1+\beta/2}}\right\|^2_{\mathcal{G}} \\
&\leq T^{\beta/2}\|Lu\|_{\mathcal{G}}\left\|\frac{u}{t^{1+\beta/2}}\right\|_{\mathcal{G}}+ \frac{(1+\delta)\alpha\beta}{2} \left\|\frac{u}{t^{1+\beta/2}}\right\|^2_{\mathcal{G}}
\end{align}
Substituting this into the  previous lemma we get
\begin{align}
\|Lu\|^2_{\mathcal{G}} &\geq \frac{\alpha \beta (1+\beta/2)}{2}\left\|\frac{u}{t^{1+\beta/2}}\right\|^2_{E^r_t,\mathcal{G}} +\|t^{-1/2}u\|^2_{E^r_t,\mathcal{G},1}\\
&-(1+\varepsilon)\left(T^{\beta/2}\|Lu\|_{\mathcal{G}}\left\|\frac{u}{t^{1+\beta/2}}\right\|_{\mathcal{G}}+ \frac{(1+\delta)\alpha\beta}{2} \left\|\frac{u}{t^{1+\beta/2}}\right\|^2_{\mathcal{G}}\right)\\
 & -C\alpha\beta\|t^{-1-\beta/2}u\|^2_{I^r_t,\mathcal{G}}-C\|t^{-1/2}u\|^2_{I^r_t,\mathcal{G},1}.
\end{align} 
Thus, if  $\delta$ and $\varepsilon$ are sufficiently small this implies
\begin{align}
\|Lu\|^2_{\mathcal{G}} &\geq \alpha \beta \delta \left\|\frac{u}{t^{1+\beta/2}}\right\|^2_{E^r_t,\mathcal{G}} +\|t^{-1/2}u\|^2_{E^r_t,\mathcal{G}_{0,1}}-(1+\varepsilon)\left(T^{\beta/2}\|Lu\|_{\mathcal{G}}\left\|\frac{u}{t^{1+\beta/2}}\right\|_{\mathcal{G}}\right)\\
& -C\alpha\beta\|t^{-1-\beta/2}u\|^2_{I^r_t,\mathcal{G}}-C\|t^{-1/2}u\|^2_{I^r_t,\mathcal{G},1}.
\end{align} 

Using the absorbing  inequality the asserted result follows, provided $T_0$ is sufficiently small and $\alpha$ is sufficiently large.
\end{proof}

In order to obtain control of the two negative terms at the right-hand side of Theorem \ref{smooth_global_Carleman_neg_term} (Global Carleman estimate with interior errors) we will use a Gaussian Carleman estimate, which is best understood in terms of the rescaled flow.

To this end choose $R\gg r$ and take 
\begin{equation}\label{tilde_u_def}
\tilde{u}=u\psi(x/\sqrt{t})
\end{equation}
where $\psi$ is $1$ up to distance $R$ and $0$ from distance $R+1$, with the bounds 
\begin{equation}\label{psi_bounds}
|\psi'| \leq 2, \qquad |\psi''|\leq 10.
\end{equation}

Set 
\begin{equation}\label{N_from_M}
N_{\tau}=\frac{M_t}{\sqrt{t}},\qquad \tau=\log(t), 
\end{equation}
and 
\begin{equation}\label{v_from_u}
\tilde{v}(y,\tau)=e^{-\tau/2}\tilde{u}(ye^{\tau/2},e^{\tau}), \qquad {v}(y,\tau)=e^{-\tau/2}{u}(ye^{\tau/2},e^{\tau}).
\end{equation}
Using the chain rule, we get that 
\begin{equation}\label{v_der_trans}
\tilde{v}_{\tau}=-\frac{1}{2}e^{-\tau/2}\tilde{u}+y\cdot \nabla \tilde{u}+e^{\tau/2}\tilde{u}_{t}, \qquad \Delta \tilde{v} = e^{\tau/2} \Delta \tilde{u},\qquad \nabla \tilde{v}= \nabla \tilde{u},
\end{equation}
so{,} by setting 
\begin{equation}\label{B_def}
B\tilde{v}:=\partial_{\tau}\tilde{v}+\Delta \tilde{v}-\frac{1}{2} y \cdot \nabla \tilde{v} +\frac{1}{2}\tilde{v},
\end{equation}
it holds
\begin{equation}\label{Op_trans}
L\tilde{u}=e^{\tau/2}B\tilde{v}. 
\end{equation}
Note that $\tilde{v}$ is supported on $B(0,R+1)$. 

\bigskip

Now set
\begin{equation}
\mathcal{F}=\exp\left(\frac{n\tau}{2}+\alpha \exp\left(-\beta \tau\right)-|y|^2/4\right), 
\end{equation}
and, similarly to before, let us denote

\begin{equation}
\|\tilde{v}\|^2_{\mathcal{F}}=\int_{-\infty}^{\log T}\int_{N_\tau}{\tilde{v}}^2\mathcal{F},\qquad \|\tilde{v}\|^2_{\mathcal{F},1}=\int_{-\infty}^{\log T}\int_{N_\tau}|\nabla \tilde{v}|^2\mathcal{F}.
\end{equation}

The following theorem provides the Carleman estimate in the core region, which leads to controlling the negative terms in Theorem \ref{smooth_global_Carleman_neg_term}.

\begin{theorem}[Gaussian Carleman at the core shrinking scale]\label{Core_Car}
Under the same conditions of Lemma \ref{first_lem}{.} For every $R<\infty${,}  there exist $T_0$ and $\alpha_0$ such that if $N_t$ is obtained from $M_t$ by \eqref{N_from_M}, $\tilde{u}$ is defined by \eqref{tilde_u_def}, $\tilde{v}$ is obtained from $\tilde{u}$ by \eqref{v_from_u}, $\alpha\geq \alpha_0$ and $T<T_0${,} then 
\begin{equation}
\|B\tilde{v}\|^2_\mathcal{F} \geq \frac{\alpha \beta^2}{16}\left\|\tilde{v}\exp\left(-\frac{\beta \tau}{2}\right)\right\|^2_{\mathcal{F}}+\frac{\beta}{32}\|\tilde{v}\|^2_{\mathcal{F},1}.
\end{equation}
\end{theorem}

\begin{proof}
Denote 
\begin{equation}
\mathcal{F}_0= \exp\left(\frac{n\tau}{2}+\alpha \exp\left(-\beta \tau\right)\right).
\end{equation}
Setting 
\[
z:=\sqrt{\mathcal{F}_0}\tilde{v}=\exp\left({\frac{n\tau}{4}+\frac{\alpha \exp(-\beta \tau)}{2}}\right)\tilde{v},
\]
we easily compute 
\[
B\tilde{v} = \mathcal{F}_0^{-1/2}\frac{\partial z}{\partial \tau}+\mathcal{F}_0^{-1/2}\Delta z-\frac{1}{2}\mathcal{F}_0^{-1/2}y\cdot \nabla z+\frac{1}{2}\mathcal{F}_0^{-1/2}z\left(1-\frac{n}{2} +\alpha \beta e^{-\beta \tau}\right)
\] 
Thus  
\begin{equation}\label{to_estimate_glob_car_lemma_gauss}
\| B\tilde{v}\|^2_{\mathcal{F}}\geq  2\left\langle \frac{\partial z}{\partial \tau}, \Delta z-\frac{1}{2}y\cdot \nabla z \right\rangle_{e^{-\frac{|y|^2}{4}}} + \left\langle \frac{\partial z}{\partial \tau}, \left(1-\frac{n}{2}+\alpha\beta e^{-\beta \tau}\right)z \right\rangle_{e^{-\frac{|y|^2}{4}}},
\end{equation}
where 
\begin{equation}
\langle z_1 , z_2 \rangle_{e^{-\frac{|y|^2}{4}}}= \int_{-\infty}^{\log T}\int_{N_\tau}z_1z_2{e^{-\frac{|y|^2}{4}}}d\mathrm{Vol}_{\tau}d\tau.
\end{equation}
For the second term on the right-hand side, observe 
\begin{align}\label{to_estimate_glob_car_lemma_1_gauss}
 \alpha\beta\left\langle \frac{\partial z}{\partial \tau}, ze^{-\beta \tau} \right\rangle_{e^{-\frac{|y|^2}{4}}}&= \frac{\alpha \beta}{2} \int \partial_\tau\int \left(z^2e^{-\beta \tau}\right)e^{-\frac{|y|^2}{4}}\\
 &+ \frac{\alpha \beta^2}{2}\left\|ze^{-\frac{\beta\tau}{2}}\right\|^2_{e^{-\frac{|y|^2}{4}}}- \frac{\alpha \beta}{2}\left\|Sze^{-\frac{\beta\tau}{2}}\right\|^2_{e^{-\frac{|y|^2}{4}}},
\end{align}
where the shrinker quantity 
\[
S:=H+\frac{1}{2}\langle y,\nu \rangle,
\]
appears as under Backwards rescaled MCF the volume form evolves by
\[
\frac{d}{d\tau}d\mathrm{Vol}_{\tau}=S^2d\mathrm{Vol}_{\tau}
\]
Assuming $T$ is small enough, it follows from the convergence of the rescaled flow to the {self-shrinker} that $|S|$ can be taken to be arbitrarily small on $B(0,R+1)$,which is where $z$ is supported. Thus,  we can estimate 
\begin{equation}
 \alpha\beta\left\langle \frac{\partial z}{\partial \tau}, ze^{-\beta \tau} \right\rangle_{e^{-\frac{|y|^2}{4}}}\geq \frac{\alpha \beta^2}{3}\left\|ze^{-\frac{\beta\tau}{2}}\right\|^2_{e^{-\frac{|y|^2}{4}}}.
\end{equation}
Estimating $\left\langle \frac{\partial z}{\partial \tau},\left(1-\frac{n}{2}\right)z \right\rangle_{e^{-\frac{|y|^2}{4}}}$ similarly, and observing it is of lower order, we get
\begin{equation}
 \left\langle \frac{\partial z}{\partial \tau}, \left(1-\frac{n}{2}+\alpha\beta e^{-\beta \tau}\right)z \right\rangle_{e^{-\frac{|y|^2}{4}}} \geq \frac{\alpha \beta^2}{4}\left\|ze^{-\frac{\beta\tau}{2}}\right\|^2_{e^{-\frac{|y|^2}{4}}}.
\end{equation}

For the first term of the {right-hand}-side of \eqref{to_estimate_glob_car_lemma_gauss}, observe that
\begin{equation}\label{first_term_bla_gauss}
2\frac{\partial z}{\partial \tau}\left(\Delta z-\frac{1}{2}y\cdot \nabla z\right)e^{-\frac{|y|^2}{4}}=2\mathrm{div}_{g(\tau)}\left(\frac{\partial z}{\partial \tau}\nabla_{g(\tau)}z e^{-\frac{|y|^2}{4}}\right)-2\left\langle\nabla_{g(\tau)}\frac{dz}{d\tau},\nabla_{g(\tau)}z\right\rangle_{e^{-\frac{|y|^2}{4}}}.
\end{equation}
Now,
\begin{equation}
2\left\langle\nabla_{g(\tau)}\frac{dz}{d\tau},\nabla_{g(\tau)}z\right\rangle = 2g^{ij}\frac{\partial^2z}{\partial x_i \partial_\tau}\frac{\partial z}{\partial x_j}=\frac{d}{d\tau}\left(g^{ij}\frac{\partial z}{\partial x_i}\frac{\partial z}{\partial x_j}\right)- \frac{dg^{ij}}{d\tau}\frac{\partial z}{\partial x_i}\frac{\partial z}{\partial x_j},
\end{equation}
which, after taking  into account the evolution of the metric, 
\begin{equation}
g'(\tau)=2SA,\qquad
\end{equation}
yields 
\begin{equation}
2\left\langle\nabla_{g(\tau)}\frac{dz}{d\tau},\nabla_{g(\tau)}z\right\rangle = 2g^{ij}\frac{\partial^2z}{\partial x_i \partial_\tau}\frac{\partial z}{\partial x_j}=\frac{d}{d\tau}g(\nabla z,\nabla z)+2SA(\nabla z,\nabla z).
\end{equation}
Substituting this into \eqref{first_term_bla_gauss}, the  divergence theorem, the evolution of the Gaussian volume,  and the boundary conditions once more we have
\begin{equation}
2\left\langle \frac{\partial z}{\partial \tau}, \Delta z-\frac{1}{2}y\cdot \nabla z \right\rangle_{e^{-\frac{|y|^2}{4}}} =  -\iint 2SA(\nabla_{g(\tau)}z,\nabla_{g(\tau)}z)e^{-\frac{|y|^2}{4}}+\iint S^2|\nabla z|^2e^{-\frac{|y|^2}{4}}.
\end{equation}

Combining these with the definition of $z$, \eqref{to_estimate_glob_car_lemma_gauss}, \eqref{to_estimate_glob_car_lemma_1_gauss}, and assuming again that $T$ is small enough so that $S$ is small on the support of $z$ yields 
\begin{equation}\label{Final_gauss_first}
\|B\tilde{v}\|^2_\mathcal{F} \geq \frac{\alpha \beta^2}{4}\left\|\tilde{v}\exp\left(-\frac{\beta \tau}{2}\right)\right\|^2_{\mathcal{F}}-\frac{\beta}{16} \|\tilde{v}\|^2_{\mathcal{F},1}.
\end{equation}

\bigskip

Now, write
\begin{equation}
\left\langle e^{-\frac{\beta \tau}{2}}\tilde{v}, e^{\frac{\beta \tau}{2}}B\tilde{v} \right\rangle_{\mathcal{F}}= \left\langle \mathcal{F}_0\tilde{v}, \frac{\partial \tilde{v}}{\partial \tau}\right\rangle_{e^{-\frac{|y|^2}{4}}} +\left\langle \mathcal{F}_0\tilde{v}, \Delta \tilde{v}-\frac{1}{2}y\nabla\ \tilde{v} +\frac{1}{2}\tilde{v} \right\rangle_{e^{-\frac{|y|^2}{4}}}.
\end{equation}
As before, integrating by parts and using the boundary conditions and the convergence to the shrinker, we get 
\begin{equation}\label{global_car_tech_est1_gauss}
\left|\left\langle \tilde{v}, \frac{\partial \tilde{v}}{\partial \tau}\right\rangle_{\mathcal{F}}\right| \leq   \alpha \beta \left\|{\tilde{v}}e^{-\frac{\beta\tau}{2}}\right\|^2_{\mathcal{F}},
\end{equation}
provided $T$ is sufficiently small. Using the divergence theorem on every time slice, we also get
\begin{align}\label{global_car_tech_est_2_gauss}
\left\langle \mathcal{F}_0\tilde{v}, \Delta \tilde{v}-\frac{1}{2}y\nabla\ \tilde{v} +\frac{1}{2}\tilde{v} \right\rangle_{e^{-\frac{|y|^2}{4}}}=-\|\tilde{v}\|_{\mathcal{F},1}^2+\frac{1}{2}\|\tilde{v}\|_{\mathcal{F}}^2.
\end{align}

\bigskip

Combining \eqref{global_car_tech_est1_gauss} and \eqref{global_car_tech_est_2_gauss} we get
\begin{align}
\|\tilde{v}\|_{\mathcal{F},1}^2 &\leq \left|\left\langle e^{-\frac{\beta \tau}{2}}\tilde{v}, e^{\frac{\beta \tau}{2}}B\tilde{v} \right\rangle_{\mathcal{F}}\right|+  \alpha \beta \left\|{\tilde{v}}e^{-\frac{\beta\tau}{2}}\right\|^2_{\mathcal{F}} \\
&\leq T^{\beta/2}\|B\tilde{v}\|_{\mathcal{F}}\left\|{\tilde{v}}e^{-\frac{\beta\tau}{2}}\right\|_{\mathcal{F}} + \alpha \beta \left\|{\tilde{v}}e^{-\frac{\beta\tau}{2}}\right\|^2_{\mathcal{F}}. 
\end{align}
Substituting this into \eqref{Final_gauss_first} we get
\begin{align}
\|B\tilde{v}\|^2_{\mathcal{F}} &\geq \frac{\alpha \beta^2}{4}\left\|\tilde{v}\exp\left(-\frac{\beta \tau}{2}\right)\right\|^2_{\mathcal{F}}+\frac{\beta}{16} \|\tilde{v}\|^2_{\mathcal{F},1}\\
&-\frac{\beta}{8}\left( T^{\beta/2}\|B\tilde{v}\|_{\mathcal{F}}\left\|{\tilde{v}}e^{-\frac{\beta\tau}{2}}\right\|_{\mathcal{F}} + \alpha \beta \left\|{\tilde{v}}e^{-\frac{\beta\tau}{2}}\right\|^2_{\mathcal{F}} \right).
\end{align} 
Using the absorbing  inequality the asserted result follows, provided $T_0$ is sufficiently small and $\alpha$ is sufficiently large.
\end{proof}

We can now prove our main technical estimate, Theorem \ref{smooth_global_Carleman} (Smooth global Carleman)

\begin{proof}[Proof of Theorem \ref{smooth_global_Carleman} (Smooth global Carleman)]
Take $R \gg r$ and note that $v=\tilde{v}$ on $B(0,R)$, and $B\tilde{v}$ is supported on $B(0,R+1)$. Thus, in light {of} the form  of the {operator} $B$ given in \eqref{B_def} and by the cutoff bounds \eqref{psi_bounds}, setting $A^R=B(0,R+1)-B(0,R)$ we have
\begin{align}
&\|Bv\|^2_{B^{R+1},\mathcal{F}}+10R^2\|v\|^2_{A^R,\mathcal{F}}+100\|v\|^2_{A^R,\mathcal{F},1} \\
 &\qquad \qquad\qquad\geq \frac{\alpha \beta^2}{16}\left\|v\exp\left(-\frac{\beta \tau}{2}\right)\right\|^2_{B^{R},\mathcal{F}}+\frac{\beta}{32}\|v\|^2_{B^{R}\mathcal{F},1}\nonumber,
\end{align}
where, as {before}, for a set $E$,  $\|f\|_{E,\mathcal{F}}$ and $\|f\|_{E,\mathcal{F},1}$ denote the corresponding integrals restricted to the set $E$. 

Scaling back to $u$ using \eqref{v_der_trans}, and denoting 
\begin{equation}
A^R_t=I^{R+1}_t-I^{R}_t, \qquad \Phi(x,t)=\exp(-|x|^2/4t),
\end{equation}
we have
\begin{align}\label{L_with_cutt}
\left\|Lu \right\|^2_{I^{R+1}_t,\Phi\mathcal{G}}&+10R^2\left\|t^{-1}u \right\|^2_{A^R_t,\Phi\mathcal{G}}+100\left\|t^{-1/2}u \right\|^2_{A^R_t,\Phi\mathcal{G},1} \nonumber\\
& \geq \frac{\alpha\beta^2}{16}\left\|t^{-1-\beta/2}u \right\|^2_{I^R_t,\Phi \mathcal{G}} +\frac{\beta}{32}\left\|t^{-1/2}u\right\|^2_{I^{R}_t,\Phi\mathcal{G}} \\
&\geq \frac{\alpha\beta e^{-\frac{r^2}{4}}}{16}\left\|t^{-1-\beta/2}u \right\|^2_{I^r_t,\mathcal{G}} +\frac{e^{-\frac{r^2}{4}}}{32}\left\|t^{-1/2}u\right\|^2_{I^{r}_t,\mathcal{G}}\nonumber,
\end{align} 
where the last inequality follows  since $R>r$ and as $\Phi$ is bounded from below by $\exp(-r^2/4)$ on $I^r_t$. Multiplying inequality \eqref{L_with_cutt} by $64 C e^{r^2/4}/\beta$, where $C$ is given in Theorem \ref{smooth_global_Carleman_neg_term} (Global Carleman estimate with interior errors), and adding the result to the inequality in that theorem, we get that

\begin{align}
&\|Lu\|^2_{\mathcal{G}} +16Ce^{r^2/4} \left[\left\|Lu \right\|^2_{I^{R+1}_t,\mathcal{G}}+R^2\left\|t^{-1}u \right\|^2_{A^R_t,\Phi\mathcal{G}}+\left\|t^{-1/2}u \right\|^2_{A^R_t,\Phi\mathcal{G},1}\right]\\
 &\geq  \alpha \beta \delta\|t^{-1-\beta/2}u\|^2_{\mathcal{G}}+ \|t^{-1/2}u\|^2_{\mathcal{G},1}.
\end{align}
Now, taking $R$ sufficiently large , the last two terms on the left-hand side can be absorbed into the right-hand side due to the $\Phi$ weight. So{,} after dividing and taking $\alpha$ sufficiently large, we get

\begin{equation}
\|Lu\|^2_{\mathcal{G}} \geq \|t^{-1-\beta/2}u\|^2_{\mathcal{G}}+ \varepsilon\|t^{-1/2}u\|^2_{\mathcal{G},1},
\end{equation}
after taking $\alpha$ sufficiently large. This completes the proof of the theorem.
\end{proof}

\begin{proof}[Proof of Theorem \ref{main_technical_uniqueness_thm}  (Backwards uniqueness for quasilinear...)]
It suffices to prove the result when $T$ is small\footnote{Note, we actually only require the short-time statement of this theorem for backwards uniqueness, as one can appeal to the existing smooth backwards uniqueness theory for mean curvature flow \cites{KS_simple,BU}.}, as then one {has shown that the flows actually agree at a later time, from which point the curvature is bounded. From this time, one can use the methods of \cite{LP} (adapted to the manifolds setting) to reach the final time of the backwards flows (the initial time of the MCFs).}

Take any $0<t_2<t_1<T$  and let $\varphi$ be a function which is $1$ for $t\in [0,t_2]$, decreasing between $t_2$ and $t_1$ and zero to $t\geq t_1$. Letting $\Omega_1$ be the domain $(M_t)_{t\in [t_2,t_1]}$ and $\Omega_2$ be the domain $(M_t)_{t\in[0,t_2]}$, setting $v=\varphi u$,  choosing $\beta\in (0,\beta_1)$ and considering the corresponding $\mathcal{G}$, we have

\begin{align}\label{tech_proof_of_tech_global_car}
\|Lv\|_{\Omega_1,\mathcal{G}}^2+\|Lu\|_{\Omega_2,\mathcal{G}}^2&=\|Lv\|^2_{\mathcal{G}}\geq   \|t^{-1-\beta/2}v\|^2_{\mathcal{G}}+\varepsilon\|t^{-1/2}v\|_{\mathcal{G},1}^2 \nonumber \\
& \geq  \|t^{-1-\beta/2}u\|^2_{\Omega_2,\mathcal{G}}+\varepsilon\|t^{-1/2}u\|_{\Omega_2,\mathcal{G},1}^2,
\end{align} 
where the first inequality uses Theorem \ref{smooth_global_Carleman} (Smooth global Carleman).

In light of \eqref{G_est}, \eqref{u_asymptotics_main}, Theorem \ref{AC_sing_struc} (1), if $T$ is sufficiently small,  we can absorb the second term in the left-hand side of \eqref{tech_proof_of_tech_global_car} into the right. Dropping the $W^{1,2}$ term, to obtain:

\begin{equation}\label{bd_norm}
\|Lv\|_{\Omega_1,\mathcal{G}}^2  \geq  \frac{1}{2}\|t^{-1-\beta/2}u\|^2_{\Omega_2,\mathcal{G}}.
\end{equation}

Now, taking any $t_3\in (0,t_2)$ and setting $\Omega_3=(M_{t})_{t\in [0,t_3]}$ we can safely use \eqref{bd_norm} to obtain
\begin{equation}\label{bd_norm_2}
\|Lv\|_{\Omega_1,\mathcal{G}}^2  \geq  \frac{1}{2}\|t^{-1-\beta/2}u\|^2_{\Omega_2,\mathcal{G}} \geq  \frac{1}{2}\|t^{-1-\beta/2}u\|^2_{\Omega_3,\mathcal{G}} \geq \|u\|^2_{\Omega_3,\mathcal{G}},
\end{equation}
provided $T$ is small enough.

Note however that 
\begin{equation}
q:= \frac{\sup_{t\in [t_2,t_1]} e^{1/t^{\beta}}}{\inf_{t\in [0,t_3]} e^{1/t^{\beta}}}<1,
\end{equation}
{so, recalling}  the definition of $\mathcal{G}$, \eqref{bd_norm_2} implies
\begin{equation}
\|u\|_{\Omega_3}^2 \leq q^{\alpha}\|Lv\|_{\Omega_1}^2\rightarrow 0, 
\end{equation}
as $\alpha \rightarrow \infty$. This completes the proof.
\end{proof}

\begin{proof}[Proof of Theorem \ref{thm_cond_back} (Conditional backwards uniqueness)]
It suffices to prove the result when $T$ is small, for then one can use the existing smooth backwards uniqueness for mean curvature flow \cite{KS_simple,BU}. By Theorem \ref{graph_over_one_another}, when $T$ is sufficiently small, one can write $M^2_t$ as a normal graph of a function $u:M^{1}_t\rightarrow \mathbb{R}$, satisfying
\begin{equation}
\frac{|u(x,t)|}{\sqrt{t}}+|{\nabla}u(x,t)|=o(1).
\end{equation} 
Recall, equation\eqref{prox_ass} assumes a much stronger decay rate for the $C^0$ norm of $u$.  
\begin{equation}\label{fast_decay}
|u(x,t)|\exp(t^{-\beta})=o(1)
\end{equation}
The Type I blowup of $|A|$ of both $M^1_t$ and $M^2_t$, which is the content of Theorem \ref{AC_sing_struc} (1), interpolated  with \eqref{fast_decay} implies that for every $\beta'\in (0,\beta)$
\begin{equation}\label{fast_decay2}
|\nabla u(x,t)|\exp(t^{-\beta'}) =o(1).
\end{equation}
Similarly, interpolating \eqref{fast_decay2} with the Type $I$ bound on $|\nabla A|$ implies that for any $\beta_1\in (0,\beta')$ we have
\begin{equation}\label{fast_decay3}
|\nabla^2 u(x,t)|\exp(t^{-\beta_1}) =o(1).
\end{equation}
 Finally, by Lemma \ref{graph_eq_evolve}, the type one behaviour of $|A|$ again, and \eqref{fast_decay}, \eqref{fast_decay2}, \eqref{fast_decay3} $u$ certainly satisfies
\[
Lu=G(x, u,\nabla u)
\]
with $G$ satisfying \eqref{G_est}, provided $T$ is sufficiently small. 
Thus, the assumption of Theorem \ref{main_technical_uniqueness_thm} (backwards uniqueness for quasilinear heat-type equations assuming fast convergence) are satisfied, and consequently $u\equiv 0$, as desired.
\end{proof}

\section{Lower bound on convergence for flows with asymptotically conical singularities}
Throughout this section, we assume that $M_t=M^1_t$ and $M^2_t$ are two BMCFs defined on $[0,T]$, both smooth outside of $(0,0)$, with AC singularity at $(0,0)$. We moreover assume that $M^2_0=M_0$. In particular, possibly after reducing $T$, Theorem \ref{graph_over_one_another} implies that  $M^2_t$ can be written as a normal graph of a function $u$ over $M_t$, satisfying 
\begin{equation}
\frac{|u(t)|}{\sqrt{t}}+|\nabla u| =o(1). 
\end{equation}
The goal of this section is to improve this rate to an exponential rate, by establishing Theorem \ref{thm_rates} (Rate of convergence).

\bigskip 

Obtaining Theorem \ref{thm_rates} (Rate of convergence) is done in three steps: we first prove a similar result with $\beta=1$ away from the singularity. We then use a localized version of an argument by Wang \cite{Wang_unique}, to obtain that the conclusion of Theorem \ref{thm_rates} (Rate of convergence) holds with some $\beta$ away from the core shrinker scale $C\sqrt{t}$. We then use an observability inequality to get the rate in the core shrinker scale.  It's important to note that the assumption of the Carleman estimate we use in the second stage (Theorem \ref{Car_rough_con|}) differs from Wang's in that it assumes local {conicality}, non-self-shrinking behaviour  as opposed to a global self shrinking asymptotically conical behaviour.

\subsection{Rate of convergence at smooth points}
Recall that $M^2_t$ is expressed as a normal graph of a function $u$ over $M_t$.
First{,} we have the following consequence of \cite[Theorem 4]{EF}, similarly to \cite[Lemma 2.1]{Nguyen} and \cite[Lemma 3.6]{Wang_unique}.
\begin{theorem}\label{exp_con_smooth}
Under the assumptions of Theorem \ref{thm_rates}, for every $\varepsilon>0$ there exists some $\delta>0$ such that on $M_t-B(0,\varepsilon)$ we have 
\begin{equation}
|u(x,t)|+|\nabla u(x,t)| \leq e^{-\delta/t} 
\end{equation} 
\end{theorem}
\begin{proof}
There exists a uniform $r=r(\varepsilon)>0$ such that for every point $p$ outside of $B(0,\varepsilon)$, there exist some orthogonal transformation $O$ such $O(M_t-p)$ can be written as a graph with small gradient over $B(0,r)=B^n(0,r)$. Moreover, the function $u$, treated as a function over $B(0,r)$ (which amounts to choosing the graphical gauge for $M_t$)  satisfies 
\begin{equation}
|\partial_{t}u+\partial_{i}(a_{ij}(x,t)\partial_j u)| \leq M(|u(x)|+|\nabla u(x)|), 
\end{equation} 
where the gradient in the right-hand side is a Euclidean gradient, and the $a_{ij}$ are uniformly elliptic, with uniform derivative bounds on the coefficients.
Now, let $\varphi$ be a spatial cutoff function which is $1$ on $B(0,r/2)$ and $0$ outside of $B(0,r)$. Given $\alpha<\infty$, let $\psi$ be a temporal cutoff function which is one for time $[0,1/\alpha]$ and $0$ for time $t\geq 2/\alpha$. Setting $w=u\psi\varphi$ we get that
\[
|\partial_{t}w+\partial_i(a_{ij}(x,t)\partial_j w)| \leq M(|w(x)|+|\nabla w(x)|)+C\alpha\chi_{E}(|u(x)|+|\nabla u(x)|),
\]
for some $C<\infty$ and where 
\[
E=B(0,r)\times [0,2/\alpha]-(B(0,r/2) \times [0,1/\alpha]). 
\]
The Carleman estimate of \cite{EF} states that there exists an $N$, independent of $\alpha$,  such that  the for every $a\in (0,1/\alpha)$.
\begin{align}\label{Car_at_reg_scale}
&\int_{B(0,r)\times [0,2/\alpha]} \sigma_a^{-\alpha}(\alpha^2 w^2+\alpha \sigma_{a}|\nabla w|^2)G_{a}dxdt\\  
\leq  &NM\int _{B(0,r)\times [0,2/\alpha]}\sigma_{a}^{1-\alpha}(w^2+|\nabla w|^2)G_adxdt+CN\alpha^2 \int _{E}\sigma_{a}^{1-\alpha}(u^2+|\nabla u|^2)G_adxdt\nonumber\\
+&a^{\alpha}N^{\alpha}\sup_{t\in [0,2/\alpha]}\int_{B(0,r)} w^2+|\nabla w|^2dx\nonumber,
\end{align}
where 
\begin{equation}\label{G_def_reg_scale}
G_a(x,t)=\frac{1}{(t+a)^{n/2}}\exp\left(-\frac{|x|^2}{4(t+a)}\right),\qquad \sigma_a(t)=\sigma(t+a),
\end{equation}
and $\sigma$ is some function satisfying 
\begin{equation}\label{sigma_growth}
N^{-1} \leq \frac{\sigma(t)}{t} \leq 1.
\end{equation}

\bigskip

Note that for every $s\in [0,2/\alpha]$,  $\sigma_{a}(s) \leq a+s \leq \frac{3}{\alpha}$. Therefore, when $\alpha\geq 2 N M$ we can absorb the first term of the right-hand side of \eqref{Car_at_reg_scale} into the left-hand side, to obtain 

\begin{align}\label{Carl_cut}
&\int_{B(0,r)\times [0,2/\alpha]} \sigma_a^{-\alpha}(\alpha^2 w^2+\alpha \sigma_{a}|\nabla w|^2)G_{a}dxdt\nonumber \\  
\leq &CN\alpha^2 \int _{E}\sigma_{a}^{1-\alpha}(u^2+|\nabla u|^2)G_adxdt+a^{\alpha}N^{\alpha}\sup_{t\in [0,2/\alpha]}\int_{B(0,r)} w^2+|\nabla w|^2dx.
\end{align}
Now, on points of in $E$ with $t\geq \frac{1}{\alpha}$ we have:
\begin{equation}\label{E_bd_1}
CN\alpha^2 \sigma(t+a)^{1-\alpha}G_a \leq CN\alpha^2 N^{\alpha-1}(t+a)^{1-\alpha-n/2} \leq 
C\alpha^{n/2+1+\alpha}N^{\alpha},
\end{equation}
where for the first inequality we have used \eqref{G_def_reg_scale} and \eqref{sigma_growth}, and for the last we have used  $t+a\geq  \alpha^{-1}$. On the portion of $E$ where $|x|\geq r/2$, as 
\begin{equation}
\max_{t\in [0,\infty)} \frac{e^{-b/t}}{t^m} =(be)^m m^m,
\end{equation}
we similarly
\begin{equation}\label{E_bd_2}
CN\alpha^2 \sigma(t+a)^{1-\alpha}G_a \leq CN\alpha^2 N^{\alpha-1}(t+a)^{1-\alpha-n/2}\exp(-r^2/16(a+t)) \leq 
N_2^{\alpha}\alpha^{n/2+1+\alpha},
\end{equation}
for some $N_2>N$ depending on $C,N,r$, whenever $\alpha$ is sufficiently large. Moreover $|w|+|\nabla w|=o(1)$, 
it follows from \eqref{Carl_cut}, \eqref{E_bd_1} and \eqref{E_bd_2} that  
\begin{equation}\label{Car_cut_2}
\int_{B(0,r)\times [0,2/\alpha]} \sigma_a^{-\alpha}(\alpha^2 w^2+\alpha \sigma_{a}|\nabla w|^2)G_{a}dxdt \leq N_2^\alpha \alpha^{n/2+1+\alpha}.
\end{equation}

\bigskip

Now, setting $\rho=1/N_2$ and choosing $a=\rho^2/(2\alpha)$, we get that in $B_{2\rho r}\times [0,\rho^2/(2\alpha)]$ we have the bound
\[
\sigma_a(t) \leq  \frac{\rho^2}{\alpha},
\]
implying
\begin{equation}
\min\{\sigma_a^{-\alpha}\alpha^2,\alpha \sigma_{a}^{-\alpha+1}\}\geq  \rho^{-2(\alpha-1)}\alpha^{2+\alpha},
\end{equation}
and so
\begin{equation}
\min\{\sigma_a^{-\alpha}\alpha^2,\alpha \sigma_{a}^{-\alpha+1}\}G_{a}\geq  \rho^{-2\alpha+2-n}\alpha^{2+\alpha+n/2}e^{-\alpha r^2}.
\end{equation}
Substituting this into \eqref{Car_cut_2} and keeping in mind that $\rho=N_2^{-1}$ we get
\begin{equation}
\int_{B_{r\rho} \times [0,\rho^2/(2\alpha)]}|u|^2+|\nabla u|^2 \leq CN_2^{-\alpha},
\end{equation}
for every  $\alpha$ sufficiently large. 
\bigskip

Substituting $t=\rho^2/(2\alpha)$, we get that  for every $t$ sufficiently small we have that there exists $\delta$ such that
\begin{equation}
\int_{B_{r\sqrt{t}} \times [0,t]}|u|^2+|\nabla u|^2 \leq e^{-\delta/t}.
\end{equation} 
Using standard interior estimates, this gives the pointwise bound as well.
\end{proof}
By scaling time, we also get a (deteriorating) version close to the singularity:
\begin{theorem}\label{exp_con_smooth_resc}
There exist $R<\infty$ and $\delta>{0}$ such that on $M_t-B(0,R\sqrt{t})$ we have 
\begin{equation}
|u(x,t)|+|\nabla u(x,t)| \leq \exp\left(-\frac{\delta |x|^2}{t}\right). 
\end{equation} 
\end{theorem}

\subsection{A Carleman estimate for roughly conical flows in a ball}
We adapt the method of \cite{Wang_unique} to obtain a Carleman estimate which is valid at points where the flow is (roughly) conical. We first quote a general Carleman estimate from that paper.

\begin{lemma}[Lemma 3.3 in \cite{Wang_unique}]\label{gen_carl}
Let $(M_t)_{t\in [0,T]}$ be a smooth BMCF on $\mathbb{R}^{n+1}-\{0\}$. Let $R>0$ and let $\mathcal{G}$ be a smooth positive function in 
\[
Q_{R,T}=\{(x,t)\;|x\in M_t-B(0,R),\;t\in [0,T]\}.
\]
Assume $v,\nabla_{M_t}v\in C^{0}_c(Q_{R,T})$ with $v(\cdot,0)=0$ and $v$ smooth in the interior of $Q_{R,T}$. Setting 
\begin{equation}\label{F_def}
\mathcal{F}=\frac{\partial_t\mathcal{G}-\Delta_{M_t}\mathcal{G}}{\mathcal{G}},
\end{equation}
we have the equality
\begin{align}
\int _0^{T}&\int_{M_t} \left(2\mathrm{Hess}\log \mathcal{G}+2HA+H^2g\right)(\nabla_{M_t}v,\nabla_{M_t}v)\mathcal{G} d\mu_tdt\\
&+\int _0^{T}\int_{M_t}\frac{1}{2}\left(\frac{d\mathcal{F}}{dt}+\Delta_{M_t}\mathcal{F}+H^2\mathcal{F}\right)v^2 \mathcal{G} d\mu_tdt\\
&\leq \int _0^{T}\int_{M_t}\left(\frac{dv}{dt}+\Delta_{M_t}v\right)^2\mathcal{G} d\mu_tdt+\int_{M_T}\left(|v|^2\mathcal{F}+|\nabla_{M_{t}}v|^2\right) \mathcal{G} d\mu_tdt.
\end{align}
\end{lemma}
\begin{proof}
While the setting in \cite{Wang_unique} is that of a self-shrinking flow, only the mean curvature flow equation is used. The proof of \cite[Lemma 3.3]{Wang_unique} carries over verbatim. 
\end{proof}

The main theorem of this section is the following Carleman estimate.

\begin{theorem}\label{Car_rough_con|}
There exists some $\delta_0\in (0,1)$ with the following significance: 

For every $C<\infty$ there exist some $\varepsilon>0$, $c<\infty$, $\alpha_0<\infty,R_0<\infty$ such that for every $R>R_0$, $r>R$ and $T\in (0,1]$ such that $c\sqrt{T} < R$, for every  $\alpha \geq \alpha_0$ and $\delta \in (\delta_0,1)$, setting  
\begin{align}\label{final}
            \mathcal{G}=\mathcal{G}_{\alpha,T}:=\mathrm{exp}\left(2\alpha(T-t)(|x|^{1+\delta}-R^{1+\delta})]+2|x|^2\right),
\end{align}
the following holds:

If $(M_t)_{t\in [0,T]}$ is a BMCF which is  $(C,c,\varepsilon)$ roughly conical in $B(0,r)$  then
\begin{align}
\int _0^{T}&\int_{M_t} \left(|\nabla_{M_t}v|^2+v^2\right) \mathcal{G} d\mu_tdt\\
&\leq \int _0^{T}\int_{M_t}\left(\frac{dv}{dt}+\Delta_{M_t}v\right)^2\mathcal{G} d\mu_tdt+\int_{M_T}|\nabla_{M_{t}}v|^2 \mathcal{G} d\mu_t,
\end{align} 
for every $v$ which vanishes on $t=0$ and outside of the annulus $B(0,r)-B(0,R)$.

\end{theorem}
\begin{proof}
By Theorem \ref{gen_carl}, it suffices to show that on $B(0,r)-B(0,R)$ we have

  \begin{align}
    &\mathcal{F}(x,T)<0 \label{Fless0} \\
    &\left(\frac{d}{dt}+\Delta_{M_t}+H^2 \right)\mathcal{F} \geq 2 \label{greater1} \\
    & 2\mathrm{Hess}\log \mathcal{G}+2HA+H^2g \geq g \label{collarlower} 
    \end{align}

We have
\begin{equation}
\frac{d \log \mathcal{G}}{dt}=-2\alpha(|x|^{1+\delta}-R^{1+\delta})+2\alpha(1+\delta)(T-t)|x|^{\delta-1}H\langle x ,\nu \rangle+4H\langle x,\nu \rangle,
\end{equation}
and
\begin{equation}
\left|\nabla_{M_t} \log \mathcal{G}\right|^2=\left(2\alpha(1+\delta)(T-t)|x|^{\delta-1}+4\right)^2|x^{\top}|^2.
\end{equation}
As
\[
\mathrm{div}_{M_t}x^{\top}=n-\langle x,\nu \rangle H,
\]
we have
\begin{equation}
\Delta_{M_t}|x|^{\beta}=\beta(\beta-2)|x|^{\beta-4}|x^{\top}|^2+\beta|x|^{\beta-2}(n-\langle x,\nu \rangle H).
\end{equation}
Thus
\begin{equation}
\Delta_{M_t}\log \mathcal{G}=2\alpha(T-t)\left[(\delta+1)(\delta-1)|x|^{\delta-3}|x^{\top}|^2+(1+\delta)|x|^{\delta-1}(n-\langle x,\nu \rangle H)\right]+4(n-\langle x,\nu \rangle H).
\end{equation}
So as 
\[
\mathcal{F}=\frac{d \log \mathcal{G}}{dt}-\Delta_{M_t}\log \mathcal{G}-\left|\nabla_{M_t} \log \mathcal{G}\right|^2,
\]
when $t=T$ we get the expression
\begin{equation}
\mathcal{F}(x,T)=-2\alpha(|x|^{1+\delta}-R^{1+\delta})+4H\langle x,\nu \rangle-4(n-\langle x,\nu \rangle)H)-16|x^{\top}|^2,
\end{equation}
which, in light of the rough conical assumption implies
\begin{equation}
\mathcal{F}(x,T) \leq -2n,
\end{equation}
provided $\varepsilon$ is sufficiently small. Thus, if $c$ is sufficiently large, we indeed get that
\[ 
\mathcal{F}(x,T)<0,
\]
establishing \eqref{Fless0}.

\bigskip

To achieve \eqref{greater1}, we collect the various terms in the definition of $\mathcal{F}$ to obtain
\begin{align}\label{F_form}
\mathcal{F}&=-2\alpha(|x|^{1+\delta}-R^{1+\delta})+2\alpha(1+\delta)(T-t)|x|^{\delta-1}H\langle x ,\nu \rangle+4H\langle x,\nu \rangle \nonumber \\
&-2\alpha(T-t)\left[(\delta+1)(\delta-1)|x|^{\delta-3}|x^{\top}|^2+(1+\delta)|x|^{\delta-1}(n-\langle x,\nu \rangle H)\right]-4(n-\langle x,\nu \rangle H)\\
&-\left(2\alpha(1+\delta)(T-t)|x|^{\delta-1}+4\right)^2|x^{\top}|^2 \nonumber
\end{align}

Observe that, in light of \eqref{F_form}, when computing the left-hand side  of \eqref{greater1}, there are terms involving $\alpha^{i}$ for $i=0,1,2$, multiplied by terms with orders of magnitudes which, in light of the rough conical assumption and standard curvature estimates, are all powers of $|x|$. We will therefore estimate the leading term (in terms of the power of $|x|$) for each of the three summands in the left-hand side of \eqref{greater1}.  With this in mind, we have
\begin{equation}
\mathcal{F} \geq -16(T-t)^2\alpha^2|x|^{2\delta}-20(T-t)\alpha|x|^{1+\delta}-20|x|^2,
\end{equation} 
so{,} by the rough conical assumption again, we can safely estimate
\begin{equation}
H^2\mathcal{F} \geq -D\left[(T-t)^2\alpha^2+D(T-t)\alpha+1\right]
\end{equation}
for some constant $D=D(R)$.

Similarly, as each derivative yields a factor of $1/|x|$ (by Definition \ref{rough_con_def}) we obtain that 
\[
\Delta \mathcal{F} \geq -D\left[(T-t)^2\alpha^2+D(T-t)\alpha+1\right]
\]
Finally, in light of the rough conical assumption  and the evolution equation for the mean curvature and $\langle x,\nu \rangle$ (See Lemma \ref{eqnofmotion}), we have that
\begin{align}
\frac{d\mathcal{F}}{dt}&=16\alpha |x|^{\delta-1}|x^{\top}|^2+8(1+\delta)^2(T-t)\alpha^2 |x|^{2\delta-2}|x^{\top}|^2+16|x^{\top}|^2+\textrm{lower order terms}\\
& \geq  8\alpha |x|^{1+\delta}+2(T-t)\alpha^2|x|^{2\delta}.\nonumber
\end{align} 
provided $\varepsilon$ is sufficiently small and $c$ and $R$ are sufficiently large. Combining the above, and taking $T$ sufficiently small, gives \eqref{greater1} readily. 

\bigskip

Finally, for \eqref{collarlower}, {n}ote that by the asymptotically conical assumption 
\[
|\nu |x|^2| \leq 2\left|\Big\langle \nu,\frac{x}{|x|}\Big\rangle\right| |x|\leq 2\varepsilon |x| ,\qquad |A|\leq \frac{C}{|x|}.
\]
Thus
\begin{equation}\label{conv_dist_squred}
\mathrm{Hess}_{M_t}|x|^2=2g +(\nu |x|^2)A \geq 1.5g,
\end{equation}
provided $\varepsilon(C)$ is sufficiently small. 

Observe,  for every $\beta$, we have
\[
\mathrm{Hess}_{M_t} f^{\beta}=\beta f^{\beta-1} \mathrm{Hess}_{M_t}f +\beta(\beta-1)f^{\beta-2}df\otimes df,
\]
and hence, taking $\beta=(1+\delta)/2$ and $f=|x|^2$ we get that 
\begin{equation}
\mathrm{Hess}_{M_t} |x|^{1+\delta}=\frac{1+\delta}{2} |x|^{\delta-1} \mathrm{Hess}_{M_t}|x|^2 +\frac{1+\delta}{2}\frac{\delta-1}{2}|x|^{\delta-3}d|x|^2\otimes d|x|^2.
\end{equation}
Note  that the non-Hessian term can be estimated by
\begin{equation}
\left|\frac{1+\delta}{2}\frac{\delta-1}{2}|x|^{\delta-3}d|x|^2\otimes d|x|^2\right| \leq 2|1-\delta||x|^{\delta-1},
\end{equation}
so{,} for $\varepsilon$ as above and $\delta$ sufficiently close to one we get 
\begin{equation}
\mathrm{Hess}_{M_t} |x|^{1+\delta} \geq |x|^{\delta-1}g \geq 0.
\end{equation}
Combining this equation, \eqref{conv_dist_squred} and the rough conical assumption, we obtain \eqref{collarlower}. This completes the proof of the theorem.

\end{proof}

\subsection{Rate of convergence below the rough conical scale}

The main theorem of this section is {that} we may extend the fast convergence of the previous section down to a  $B(0,C\sqrt{t})$ neighborhood of the origin. 

\begin{theorem}\label{con_cone}
There exists $R<\infty$ and $\beta\in (0,1/2)$ such that there exists $t_0$ such that for all $t<t_0$ and every point in $M_t-B(0,R\sqrt{t})$ we have
\begin{equation}\label{rate_outside_C_ball}
|u(x,t)|+|\nabla u(x,t)| \leq e^{-1/t^{\beta}} 
\end{equation} 
\end{theorem}
\begin{proof}
By item (3) of Theorem \ref{AC_sing_struc} (Structure of flow with AC singularity), there exist some $S>0$ and $C<\infty$ such that for every $\varepsilon>0$ there exist $r>0$ and $R_1>0$ such that that for every $\lambda\geq 1$ the flow $\lambda M^i_{\lambda^{-2} t}$ for $t\in [0,\lambda^2 S]$ are $(C,R_1,\varepsilon)$ roughly conical in $B(0,2\lambda r)$. Let $\varepsilon>0$ $c<\infty$, $\alpha_0<\infty$ and $R_0<\infty$ be the constants, depending on $C$, which are given by Theorem \ref{Car_rough_con|}. With this choice of $\varepsilon=\varepsilon(C)$, and the corresponding $r_1,R_1$ as above, we get that given any time $t_0\in [0,S]$ setting 
\[
N^i_t:=\frac{1}{\sqrt{t_0}}M^i_{\frac{t}{t_0}},
\]
the flows $(N^i_t)_{t\in (0,1]}$ are $(C,R_0,\varepsilon)$ roughly conical in $B(0,2r/\sqrt{t_0})$. Taking 
\[
R=\max\{R_0,R_1,c\}, 
\] 
the $(N^i_t)_{t\in (0,1]}$ are clearly $(C,R,\varepsilon)$ roughly conical in $B(0,2r/\sqrt{t_0})$, provided $t_0$ is sufficiently small. Note that this choice of an inner radius $R$, as well as the choice of an outer radius $2r/\sqrt{t_0}$ is consistent with the condition imposed by Theorem \ref{Car_rough_con|}, the $r$ therein being $2r/\sqrt{t_0}$, provided $t_0$ is sufficiently small. Note further that on $B(0,2r/\sqrt{t_0})-B(0,R)$ the graph function $v$ of $N^2_t$ over $N_t:=N^1_t$ satisfies 
\begin{equation}\label{v_err_bd}
|\partial_t v+\Delta^{N^1_t} v| \leq \varepsilon (|v|+|\nabla v|),
\end{equation}
again, provided that $t_0$ is sufficiently small. 

\bigskip

Take a spatial cutoff function $\varphi:\mathbb{R}_{+}\rightarrow [0,1]$ which is $0$ for $|x|\leq R$ and for $|x|\geq 2r/\sqrt{t_0}$ and $1$ for $|x|\in [(2R, r/\sqrt{t_0}]$ with derivatives estimated by 
\begin{equation}\label{cut_spa_bound}
R|D\varphi|+R^2|D\varphi|^2 \leq C_1.
\end{equation}
Let us further choose a temporal cutoff function $\psi$ which is $1$ for $t\geq 2\varepsilon$ and $0$ for $t\leq \varepsilon$ with 
\begin{equation}\label{cut_temp_bound}
|\psi'| \leq \frac{2}{\varepsilon}.
\end{equation}
Considering the function 
\begin{equation}
w=v\psi{\varphi},
\end{equation}
and using \eqref{v_err_bd}, \eqref{cut_spa_bound} and \eqref{cut_temp_bound}, there exists some $C_2=C_2(R)<\infty$, such that 
\begin{equation}\label{w_err_est}
|\partial_t w+\Delta w|^2 \leq \frac{1}{2}(|w|^2+|\nabla w|^2)+C_2(|v|^2+|\nabla v|^2)\chi_{E_1}+\frac{C_2}{\varepsilon^2}\chi_{E_2}
\end{equation}
where 
\begin{equation}\label{E1_def}
E_1 = \left(\left(B(0,2r/\sqrt{t_0})-B(0,r/\sqrt{t_0})\right) \cup \left(B(0,2R)-B(0,R)\right)\right)\times [0,2\varepsilon],
\end{equation}
\begin{equation}\label{E2_def}
E_2= \left(B(0,2r/\sqrt{t_0})-B(0,R)\right)\times [\varepsilon,2\varepsilon].
\end{equation}

\bigskip

Since  our choice of constants at the beginning of the proof satisfy the assumption of Theorem \ref{Car_rough_con|}, and since $w$ vanishes at $t=0$ and outside of $B(0,2r/\sqrt{t_0})-B(0,R)$, Theorem \ref{Car_rough_con|} is applicable for $w$. Thus, letting $\delta_0$ be as in that theorem and choosing any $T\in [0,1]$, $\alpha\geq \alpha_0$, and  $\delta\in (\delta_0,1)$ we have that 
\begin{align}\label{car_for_w}
\int _0^{T}&\int_{N_t} \left(|\nabla_{M_t}w|^2+w^2\right) \mathcal{G} d\mu_tdt\\
&\leq \int _0^{T}\int_{N_t}\left(\frac{dw}{dt}+\Delta_{N_t}w\right)^2\mathcal{G} d\mu_tdt+\int_{N_T}|\nabla_{N_{t}}w|^2 \mathcal{G} d\mu_tdt,
\end{align}
where $\mathcal{G}$ is given in \eqref{final}. 

Substituting \eqref{w_err_est} into the first term in the right-hand side of \eqref{car_for_w}, and absorbing the first resulting term into the left-hand side of \eqref{car_for_w}, we obtain\footnote{recalling the domain definitions \eqref{E1_def} and \eqref{E2_def}}

\begin{align}\label{car_for_w2}
\int_{0}^{T}\int_{N_t}|w|^2 \mathcal{G}\leq &C_3\int_{0}^{T}\int_{N_t\cap  \left(B(0,2R)-B(0,R)\right)}(v^2+|\nabla v|^2)\mathcal{G}  \nonumber \\ 
+&C_3\int_{0}^{T}\int_{N_t\cap  \left(B(0,2r/\sqrt{t_0})-B(0,r/\sqrt{t_0})\right)}(v^2+|\nabla v|^2)\mathcal{G} \\
+& C_3\int_{N_{T}}|\nabla w|^2 \mathcal{G}\nonumber \\
+& \frac{C_3}{\varepsilon^2}\int_{\varepsilon}^{2\varepsilon} \int_{N_t\cap  \left(B(0,2r/\sqrt{t_0})-B(0,R)\right)}(v^2+|\nabla v|^2)\mathcal{G}   \nonumber,
\end{align}
for some $C_3=C_3(C_2)$.

\bigskip

Since the flows $N_t$ and $N^2_t$  converge to one another smoothly outside of $B(0,R)$, there exists some $D<\infty$ such that outside of $B(0,R)$, 
\[
|v(x,t)|+|\nabla v(x,t)| \leq Dt \qquad x\notin B(0,R).
\]
Thus{,} by the dominated convergence theorem, we can pass \eqref{car_for_w2} to the limit as $\varepsilon\rightarrow 0$. Re-defining $w$ to be 
\begin{equation}
w=v\phi,
\end{equation}
we thus obtain 
\begin{align}\label{Lu_Carl_app}
\int_{0}^{T}\int_{N_t}|w|^2 \mathcal{G} \leq &C_3\int_{0}^{T}\int_{N_t\cap  \left(B(0,2R)-B(0,R)\right)}(v^2+|\nabla v|^2)\mathcal{G} \nonumber  \\
+&C_3\int_{0}^{T}\int_{N_t\cap  \left(B(0,2r/\sqrt{t_0})-B(0,r/\sqrt{t_0})\right)}(v^2+|\nabla v|^2)\mathcal{G}   \\
+& C_3\int_{N_{T}}|\nabla w|^2 \mathcal{G}. \nonumber
\end{align}
We will now lower bound the left-hand side and upper bound the terms in the right-hand side of \eqref{Lu_Carl_app}. As a first step, we will derive bounds for $\mathcal{G}$ in the domains of integration appearing in \eqref{Lu_Carl_app}. For the reader's convenience, we recall here that
\begin{align}\label{final_recall}
            \mathcal{G}=\mathcal{G}_{\alpha,T}:=\mathrm{exp}\left(\alpha(T-t)(|x|^{1+\delta}-R^{1+\delta})]+2|x|^2\right).
\end{align}
First, observe that for $(x,t)\in \left(\mathbb{R}^{n+1}-B(0,17R)\right) \times (0,T/2]$ we have  
\begin{equation}\label{G_estimate_T/2}
\mathcal{G}\geq  \exp\left(\frac{\alpha T}{2}(17^{1+\delta}-1)R^{1+\delta}\right),
\end{equation}
while for $(x,t)\in \left(B(0,2R)-B(0,R)\right)\times (0,T]$ we similarly have
\begin{equation}\label{G_estimate_anu}
\mathcal{G}\leq  \exp\left(\alpha T(2^{1+\delta}-1)R^{1+\delta}+8R^2\right),
\end{equation}
and for $(x,t)\in B(0,2r/\sqrt{t_0})-B(0,r/\sqrt{t_0})\times (0,T]$ we have 
\begin{equation}\label{G_estimate_away}
\mathcal{G}\leq  \exp\left(\alpha (2r/\sqrt{t_0})^{1+\delta}+8r^2/t_0\right).
\end{equation}
For {the }last weight estimate, note that  on $N_T$, along the support of $w$ we have
\begin{equation}\label{G_estimate_T}
\mathcal{G} = \exp\left(2|x|^2 \right).
\end{equation}

\bigskip

By Theorem \ref{exp_con_smooth_resc}, there exist $\eta>0$ such  for any $t\in (0,T]$ and $x\in N_t$ we have 
\begin{equation}\label{vw_estimate_smooth}
|v|^2+|\nabla v|^2+|w|^2+|\nabla w|^2 \leq \exp\left(-\frac{\eta|x|^2}{t}\right) 
\end{equation}

Thus, if we choose $T$ with $T<\eta/16$, combining \eqref{vw_estimate_smooth} and \eqref{G_estimate_T}  we get 
\begin{equation}\label{all_estimate_T}
\int_{N_{T}}|\nabla w|^2 \mathcal{G} \leq \int_{N_T} \exp(2|x|^2-16|x|^2) \leq D. 
\end{equation}

Similarly, combining \eqref{vw_estimate_smooth} with \eqref{G_estimate_away} we have 
\begin{equation}\label{all_estimate_away}
\int_{0}^{T}\int_{N_t\cap  \left(B(0,2r/\sqrt{t_0})-B(0,r/\sqrt{t_0})\right)}(v^2+|\nabla v|^2)\mathcal{G} \leq \exp\left(\alpha (2r/\sqrt{t_0})^{1+\delta}-8r^2/t_0\right).
\end{equation}

\bigskip

Substituting \eqref{all_estimate_T} and \eqref{all_estimate_away} into the last to terms on of the left-hand side of \eqref{Lu_Carl_app}, plugging the $\mathcal{G}$ bounds  \eqref{G_estimate_T/2} and \eqref{G_estimate_anu} for the remaining two terms and using \eqref{vw_estimate_smooth} to estimate $v^2+|\nabla v|^2\leq D$ on $B(0,2R)-B(0,R)$ we thus get 
\begin{align}\label{exp_both_sides_calc}
&\exp\left(\frac{\alpha T}{2}(17^{1+\delta}-1)R^{1+\delta}\right) \int_{0}^{T/2}\int_{N_t \cap \left(B(0,r/\sqrt{t_0})-B(0,17R)\right)}|v|^2 \\
&\qquad\leq C_4 \left(\exp\left(\alpha T(2^{1+\delta}-1)R^{1+\delta}\right)+ \exp\left(\alpha (2r/\sqrt{t_0})^{1+\delta}-8r^2/t_0\right)\right).\nonumber
\end{align}
Note that the exponent on the left-hand side of \eqref{exp_both_sides_calc} is larger than then first exponent in the right-hand side by a factor of more than $2$, so we obtain 
\begin{align}
& \int_{0}^{T/2}\int_{N_t \cap \left(B(0,r/\sqrt{t_0})-B(0,17R)\right)}|v|^2 \\
&\qquad\leq C_4\exp\left(-\alpha TR^{1+\delta}\right) \left(1+ \exp\left(\alpha (2r/\sqrt{t_0})^{1+\delta}-r^2/t_0\right)\right).\nonumber
\end{align}
Finally, setting
\[
\alpha=\frac{1}{2}\left(\frac{2r}{\sqrt{t_0}}\right)^{1-\delta},
\]
we finally obtain that 
\begin{equation}
\int_{0}^{T/2}\int_{N_t \cap \left(B(0,r/\sqrt{t_0})-B(0,17R)\right)}|v|^2 \leq C_5\exp\left(-\theta/{t_0^{(1-\delta)/2}}\right),
\end{equation}
for some $\theta>0$ and some $C_5<\infty$. Converting this integral estimate into pointwise $C^1$ estimate follows again from standard parabolic estimates.
\end{proof}

\subsection{Rate of convergence at the core shrinker scale}\label{sec_core}
The purpose of this section is to complete the proof of Theorem \ref{thm_rates} which, in light of the previous sections, amounts to extending the convergence rate \eqref{rate_outside_C_ball}achieved outside of $B(0,R\sqrt{t})$, into that parabolic ball.

\begin{proof}[Proof of Theorem \ref{thm_rates} (Rate of convergence)]
By Theorem \ref{con_cone}, we know that there exist some $C<\infty,T>0$ and $\beta\in (0,1)$ such that 
\begin{equation}\label{u_forth_bound}
|\nabla^i u(x,t)| \leq \exp\left(-\frac{1}{t^{\beta}}\right), \qquad t\in (0,T),\; x\in M_t-B(0,C\sqrt{t}).
\end{equation}
Consider the rescaled BMCF
\begin{equation}
 N_{\tau}=\frac{M_t}{\sqrt{t}},\; N^2_{\tau}=\frac{M^2_t}{\sqrt{t}} \qquad t=e^{\tau},  
\end{equation}
where $\tau\in (-\infty,\log(T))$. Expressing $N^2_{\tau}$ as a graph of a function $w$ over  $N_{\tau}$, $w$ is related to $u$ by  
\[
w(x,\tau):=e^{-\tau/2}u(xe^{\tau/2},e^{\tau})
\]
and so \eqref{u_forth_bound} implies (after reducing $\beta$ and $T$)
\begin{equation}\label{w_ass}
|\nabla^i w(x,\tau)| \leq \exp\left(-\exp(-\tau \beta)\right), \qquad \tau \in (-\infty,\log(T)),\; x\in N_\tau-B(0,C).
\end{equation}

\bigskip

We regard \eqref{w_ass} as  Dirichlet \textit{and} Neumann boundary conditions, whose smallness should be translatable to smallness in $N_{\tau}-B(0,C)$ via methods from the theory of inverse problems. Such estimates are often phrased for Euclidean domain, and their proof relies on this Euclidean structure as well. As such,  we need an estimate can conclude interior smallness from Dirichlet  and Neumann smallness along only a  \textit{portion} of the boundary.  The precise theorem we employ is the following stability result which appears in \cite{Isakov}.

\begin{theorem}[Theorem 3.3.10 in \cite{Isakov}]\label{isak_thm}
Let $D\subseteq \mathbb{R}^n$ be a domain with smooth boundary, and let $\Gamma\subseteq \partial D$.  Suppose $w:\overline{D}\times [0,3]\rightarrow \mathbb{R}$  satisfies an equation of the  form 
\begin{equation}\label{back_par}
-w_{\tau}=a^{ij}(x,t)\partial^2_{x_ix_j}w+b_i(x,t)\partial_{x_i}w+c(x,t)w=0,
\end{equation} 
where the coefficients are uniformly bounded in $C^2$ and are uniformly elliptic.  Suppose further that $E\subseteq D$ is an open domain  such that $\overline{E} \subseteq D\cup \Gamma$.

Then there exist some $C<\infty$ and $\alpha\in (0,1)$, depending on $E,D$, and bounds on the coefficients of the equation \eqref{back_par} and its ellipticity, with the following significance: 

If $w$ is a solution of \eqref{back_par} on $\overline{D}\times [0,3]$, satisfying
\begin{equation}
w(x,t)=f,\; \partial_{\nu}w(x,t)=g,\qquad (x,t)\in \Gamma\times [0,3]
\end{equation}
then
\begin{equation}
\|w\|_{L^2(E\times [1,2])} \leq C\left(F+\|w\|_{L^2(D\times [0,3])}^{1-\alpha} F^{\alpha}\right),
\end{equation}
where
\[
F=\|f\|_{W^{1,2}(\Gamma\times [0,3])}+\|g\|_{L^2(\Gamma \times [0,3])}
\]
\end{theorem}

To apply this theorem in the setting at hand, note that since $N_{\tau}\rightarrow \Sigma$ as $\tau\rightarrow -\infty$, for every $\tau$ sufficiently negative there exists a domain $\Omega_{\tau} \subseteq \Sigma$ such that $N_{\tau}\cap B(0,5C)$ is a normal graph over $\Omega_{\tau}$. Denoting the above normal parametrization by  $\psi_{\tau}:\Omega_{\tau}\rightarrow N_{\tau}\cap B(0,5C)$ we further have that 
\begin{equation}\label{co_ord_cover}
\overline{B(0,4C)}\cap \Sigma \subseteq \Omega_{\tau},\qquad \textrm{and}\; \qquad N_{\tau}\cap \overline{B(0,2C)} \subseteq \psi_{\tau}(B(0,3C)\cap \Sigma)  .
\end{equation}

We now construct a cover for $B(0,3C)$ of smooth disks,  and sub-disks, which will serve as the domains $D$ and $E$ in Theorem \ref{isak_thm}. 

To this end, pick a point $p_0\in \partial{B(0,3C)}\cap \Sigma$.  For every point $p\in \overline{B(0,3C)}\cap \Sigma$ there exists some $\varepsilon_p$ and smooth closed \textbf{disks} $\overline{E}_p \subseteq \overline{D_{p}} \subseteq \Sigma \cap B(0,4C)$ such that 
\begin{enumerate}
\item $B(p,\varepsilon_p) \subseteq \mathrm{Int}(E_{p})$.
\item $p_0\in E_p$.
\item The set $\Gamma_p:= \partial B(p_0,\varepsilon_p)\cap \partial D_{p}$
is a half sphere of radius $\varepsilon_i$ around $p_0$ and furthermore
\[
\Gamma_p=\partial B(p_0,\varepsilon_p)\cap \partial E_{p}=\partial D_p\cap \partial E_p. 
\]
\end{enumerate} 
Such $D_p$  can be explicitly constructed by considering a simple, smooth curve between $p_0$ and any point $p\in \overline{B(0,3C)}$ and then choosing an $\varepsilon_p>0$ such that $2\varepsilon_p$-neighborhood of the curve is a disk, and then smoothing out the $\varepsilon_p$ neighborhood of this curve. Pushing $D_p$ inward a bit, away from $\Gamma_p$, gives the desired $E_p$. 

By compactness, there exist some $p_1,\ldots p_k$ such that the $E_i:=E_{p_i}$ cover $\overline{B(0,3C)}\cap \Sigma$. Denote   $D_i=D_{p_i}$
\[
\Gamma_i:=\partial B(p_0,\varepsilon_{p_{i}})\cap \partial D_i.
\]

\bigskip

Now, for each $i=1,\ldots k$, considering $\overline{D_i}$ with the pullback metric $\psi_{\tau}^{\ast}g_{\mathrm{Euc}}$ and considering $w$ in this parametrization (so we identify $w$ with $w\circ \psi_{\tau}$) it follows that $w$ satisfies a quasilinear backwards parabolic equation on $\overline{D_i}$. In particular, as $|\nabla^j w|$ converges to zero for $j=0,\ldots 4$ as $\tau\rightarrow -\infty$,  $w$ satisfies a linear equation of the form \eqref{back_par}, where the coefficients are uniformly bounded in $C^2$ and are uniformly elliptic.

\bigskip

We are now in a position to apply Theorem \ref{isak_thm}  to (the representation of) the function $w$ over the domains $D_i$, and its time translates. Since $w\rightarrow 0$ as $\tau\rightarrow -\infty$ on each $D_i$ we get that there exist constants $C_i$ and $\alpha_i\in (0,1)$ such that 
\begin{equation}
\|w\|_{L^2(E_i\times [\tau,\tau+1])} \leq C_i\left(F_i+F_i^{\alpha_{i}}\right),
\end{equation}
where 
\begin{equation}
F_i=\|w\|_{W^{1,2}(\Gamma_i\times [\tau-1,\tau+2])}+\|\partial_{\nu}w\|_{L^2(\Gamma_i \times [\tau-1,\tau+2])}
\end{equation}
In light of \eqref{w_ass}, this implies that 
\begin{equation}
\|w\|_{L^2(E_i\times [\tau,\tau+1])} \leq C_i\exp(-\alpha_i\exp(-\beta \tau)),
\end{equation}
where $C_i$ and $\alpha_i$ were updated from the previous step. Since the $E_i$ cover $\overline{B(0,3C)\cap \Omega}$, using \eqref{co_ord_cover} we thus get that there exist some $D$ and $\alpha>0$ such that 
\begin{equation}
\|w\|_{L^2((N_{\ast}\cap B(0,2C))\times [\tau,\tau+1])} \leq D\exp(-\alpha\exp(-\beta \tau)),
\end{equation}
Using standard parabolic estimates, this implies that (after changing $D$ and $\alpha$)
\begin{equation}
\|\nabla^i w\|_{B(0,C)\cap N_\tau} \leq D\exp(-\alpha\exp(-\beta \tau)).
\end{equation}
Chasing this estimate back through the re-parameterisations and dilation to $u$ gives the desired result.
\end{proof}

\bibliography{epi}
\bibliographystyle{alpha}

\end{document}